\documentclass[12pt]{amsart}
\usepackage{graphicx}
\usepackage{esint}
\usepackage{amssymb, mathrsfs, url, amsfonts, amsthm, amsmath}
\usepackage{enumerate}
\usepackage{xcolor}
\usepackage{etoolbox}
\apptocmd{\sloppy}{\hbadness 10000\relax}{}{}
\apptocmd{\sloppy}{\vbadness 10000\relax}{}{}
\usepackage{hyperref}
\usepackage[letterpaper,margin=1.1in]{geometry}

\newtheorem{introthm}{Theorem}

\newtheorem{theorem}{Theorem}[section]
\newtheorem{lemma}[theorem]{Lemma}
\newtheorem{corollary}[theorem]{Corollary}
\newtheorem{proposition}[theorem]{Proposition}
\newtheorem{observation}[theorem]{Observation}
\def\lec{\lesssim}

\theoremstyle{definition}
\newtheorem{definition}[theorem]{Definition}

\theoremstyle{remark}
\newtheorem{remark}[theorem]{Remark}

\newcommand{\VMO}{\mathrm{VMO}}
\newcommand{\BMO}{\mathrm{BMO}}

\newcommand{\gap}{\mathop\mathrm{gap}\nolimits}

\let\inf\relax \DeclareMathOperator*\inf{\vphantom{p}inf}

\newcommand{\Haus}{\mathcal{H}}
\newcommand{\RR}{\mathbb{R}}

\numberwithin{equation}{section}
\numberwithin{figure}{section}
\newcommand{\R}{\mathbb R}

\def \colonequals {\mathrel{\mathop:}=}

\newcommand{\excess}{\mathop\mathrm{excess}\nolimits}
\newcommand{\dist}{\mathop\mathrm{dist}\nolimits}

\newcommand{\spt}{\mathop\mathrm{spt}\nolimits}

\newcommand{\diam}{\mathop\mathrm{diam}\nolimits}

\def \colonequals {\mathrel{\mathop:}=}

\def\Xint#1{\mathchoice
{\XXint\displaystyle\textstyle{#1}}%
{\XXint\textstyle\scriptstyle{#1}}%
{\XXint\scriptstyle\scriptscriptstyle{#1}}%
{\XXint\scriptscriptstyle\scriptscriptstyle{#1}}%
\!\int}
\def\XXint#1#2#3{{\setbox0=\hbox{$#1{#2#3}{\int}$ }
\vcenter{\hbox{$#2#3$ }}\kern-.6\wd0}}

\def\dashint{\Xint-}

\def\res{\hbox{ {\vrule height .22cm}{\leaders\hrule\hskip.2cm} } }

 \numberwithin{equation}{section}

\title[Two-phase problem for harmonic measure with H\"older data]{Regularity of the singular set in a two-phase problem for harmonic measure with H\"older data}
\date{June 1, 2019}
\author{Matthew Badger}
\author{Max Engelstein}
\author{Tatiana Toro}
\thanks{M.~Badger was partially supported by NSF grants DMS 1500382 and DMS 1650546. M.~Engelstein was partially supported by an NSF postdoctoral fellowship, NSF DMS 1703306 and NSF DMS 1500771. T.~Toro was partially supported by NSF grants DMS 1361823 and DMS 1664867 and by the Craig McKibben \& Sarah Merner Professorship in Mathematics. This material is based upon work supported by the National Science Foundation under Grant No. DMS-1440140 while the authors were in residence at the Mathematical Sciences Research Institute in Berkeley, California, during the Spring 2017 semester.}
\address{Department of Mathematics\\ University of Connecticut\\ Storrs, CT 06269-3009}
\email{matthew.badger@uconn.edu}
\address{Department of Mathematics\\Massachusetts Institute of Technology\\Cambridge, MA, 02139-4307 }
\email{maxe@mit.edu}
\address{Department of Mathematics\\ University of Washington\\ Box 354350\\ Seattle, WA 98195-4350}
\email{toro@uw.edu}
\begin{document}

\begin{abstract} In non-variational two-phase free boundary problems for harmonic measure, we examine how the relationship between the interior and exterior harmonic measures of a domain $\Omega\subset\RR^n$ influences the geometry of its boundary. This type of free boundary problem was initially studied by Kenig and Toro in 2006 and was further examined in a series of separate and joint investigations by several authors. The focus of the present paper is on the singular set in the free boundary, where the boundary looks infinitesimally like zero sets of homogeneous harmonic polynomials of degree at least 2.  We prove that if the Radon-Nikodym derivative of the exterior harmonic measure with respect to the interior harmonic measure has a H\"older continuous logarithm, then the free boundary admits unique geometric blowups at every singular point and the singular set can be covered by countably many $C^{1,\beta}$ submanifolds of dimension at most $n-3$. This result is partly obtained by adapting tools such as Garofalo and Petrosyan's Weiss type monotonicity formula and an epiperimetric inequality for harmonic functions from the variational to the non-variational setting. \end{abstract}

\maketitle
\setcounter{tocdepth}{1}
\tableofcontents

\section{Introduction}\label{introandprelim}

Harmonic measure is a canonical measure, associated to any domain $\Omega\subset\RR^n$ ($n\geq 2$), which arises naturally via the solution of the classical Dirichlet problem. The measure is supported on (a subset of) the boundary of the domain. For a comprehensive introduction, see \cite{CKL}, \cite{GM}. In this paper, we continue to investigate the strong connection between analytic regularity of the harmonic measure and geometric regularity of the boundary of the domain in the two-phase setting. Thus, assume that $\Omega^+=\Omega\subset\RR^n$ is a domain with nonempty, connected exterior $\Omega^-=\RR^n\setminus\overline{\Omega}$. We refer to $\Omega^+$ as the interior domain and  to $\Omega^-$ as the exterior domain. Let $\omega^+$ and $\omega^-$ denote the harmonic measures on $\Omega^+$ and $\Omega^-$, respectively. If $\omega^+$ and $\omega^-$ are mutually absolutely continuous, then $\partial\Omega^+$ and $\partial\Omega^-$ coincide up to a set of mutual interior and exterior harmonic measure zero and we may form the Radon-Nikodym derivative of $\omega^-$ with respect to $\omega^+$, $$h\equiv \frac{d\omega^-}{d\omega^+}:\partial\Omega\rightarrow[0,+\infty].$$ The Radon-Nikodym theorem ensures that $h\in L^1(d\omega^+)$ for every pair of admissible domains $\Omega^+$ and $\Omega^-$, including \emph{a priori} domains with disconnected or irregular boundaries. The \emph{two-phase free boundary regularity problem for harmonic measure} is to determine the extent to which existence and/or additional control on $h$ limits the geometry of $\partial\Omega^+\cap\partial\Omega^-$.

Following \cite{kenigtorotwophase}, \cite{engelsteintwophase}, and \cite{BETHarmonicpoly}, it is known that if $\Omega^+$ and $\Omega^-$ are NTA domains (see \S\ref{s: notation}) and $\log h$ is an $\alpha$-H\"older continuous, real-valued function, then $$\partial\Omega=\Gamma_1\cup S,$$ where $\Gamma_1$ is an $(n-1)$-dimensional $C^{1,\alpha}$ embedded submanifold and $S$ (the ``singular set") is a closed set of Hausdorff and Minkowski dimension at most $n-3$. The goal of this paper is to strengthen our understanding of the singular set in the H\"older continuous regime. We prove that $\partial\Omega$ has unique blowups at points in $S$ (see Theorem \ref{t:unique}) and furthermore establish higher order, $C^{1,\beta}$ rectifiability of $S$ (see Theorem \ref{t:main}).

Our basic strategy is to use a Weiss type functional $W_d(r,u)$ in conjunction with an epiperimetric inequality for harmonic functions (for an overview, see \S\ref{s: WeissMonotonicity}). These are well known tools in the calculus of variations, but our use of them in this paper is novel as there is no underlying energy in our problem. Indeed, the utility of functionals like $W_d(r,u)$ is that they are monotone increasing in $r$ when $u$ minimizes an associated energy (see Garofalo and Petrosyan \cite{garofaloandpetrosyan}, who introduced $W_d(r,u)$ to study the thin obstacle problem). However, in the context of our two-phase problem for harmonic measure, we must analyze $W_d(r,v)$ for certain functions $v$ (built from $h$ and the Green's functions $u^\pm$ associated to $\Omega^\pm$), which are not minimizers, and there is no reason to expect that $W_d(r,v)$ is monotone in $r$. To overcome this deficit, we must build $v$ carefully and use a precise description of the local dimension of harmonic measure at singular points (see \S\ref{s: notation}) to prove that the functions $v$ are ``almost harmonic'' in the sense of distributions. The fact that $v$ is ``almost harmonic" then allows us to bound the growth of $W_d(r,v)$ (see \S\ref{s:boundW}) and establish regularity of the singular set $S$ (see \S\S\ref{s: c1betaconvergence} and \ref{s: rectifiability}).  Throughout, the possibility of degeneracy is the main difficulty (see \S\ref{s: WeissMonotonicity} for more discussion). We deal with this first by working with Almgren's frequency functional (see \S\ref{formulasandcomputations}) and later by using the aforementioned growth of $W_d(r,v)$ to prove that degeneracy does not occur (in this approach we are motivated by work of Garofalo, Petrosyan, and Smit Vega Garcia \cite{gpsvgepiperimetric}).

\subsection{Background and statement of main results}\label{s:background}

Azzam, Mourgoglou, Tolsa, and Volberg recently resolved a long-standing conjecture of Bishop \cite{Bishop-questions} on the two-phase problem for harmonic measure on domains in space. They proved the following:

\begin{introthm}[see Azzam et. al. {\cite{amtv-twophase}}] \label{intro:a} Let $\Omega^+=\Omega\subseteq\RR^n$ and  $\Omega^-=\RR^n\setminus\overline{\Omega}$ be complementary domains in $\RR^n$, $n\geq 3$, equipped with harmonic measures $\omega^\pm$, respectively. Assume that $\omega^+\ll\omega^-\ll\omega^+$. Then $\omega^\pm$ is Lipschitz graph rectifiable: $\omega^\pm(\RR^n\setminus G)=0$, where $G= \bigcup_{i=1}^\infty \Gamma_i$ for some $\Gamma_i\subseteq\RR^n$, which are isometric copies of $(n-1)$-dimensional Lipschitz graphs in $\RR^n$; moreover $\omega^\pm \res G\ll \Haus^{n-1}\res G \ll \omega^\pm \res G$, where $\Haus^{n-1}$ denotes codimension one Hausdorff measure and $\mu\res A$ denotes a measure $\mu$ restricted to a set $A$.\end{introthm}

For related prior work, see \cite{kenigpreisstoro},  \cite{amt-twophase}. Also see the recent preprint \cite{amt-ur}, which explores quantitative conditions on $\omega^\pm$ that ensure $\partial\Omega$ contains $\omega^\pm$ big pieces of uniformly rectifiable sets. We note that the conclusion of Theorem A makes no assertion about the dimension of the $\omega^\pm$ null set $\partial\Omega\setminus G$. In principle, the Hausdorff or Minkowski dimension of $\partial\Omega\setminus G$ could exceed  $n-1$. See \cite[\S2]{GarnettOFarrell} for an example, with $n=2$, where $\partial\Omega\setminus G$ has positive $\Haus^1$ measure.

By imposing further restrictions on the relationship between interior harmonic measure $\omega^+$ and exterior harmonic measure $\omega^-$ than in Theorem A, one is able to say more about the fine geometry of the free boundary:

\begin{introthm}[see Badger, Engelstein, Toro {\cite{BETHarmonicpoly}}] \label{intro:b} In addition to the hypothesis of Theorem \ref{intro:a}, assume $\Omega^+$ and $\Omega^-$ are NTA domains and the Radon-Nikodym derivative $h=d\omega^-/d\omega^+$ satisfies $\log h\in \VMO(d\omega^+)$ or $\log h\in C(\partial\Omega)$.  Then: \begin{enumerate}
\item There exist $d_0\geq 1$ depending on at most $n$ and the NTA constants of $\Omega^+$ and $\Omega^-$ such that $\partial\Omega$ is locally bilaterally well approximated (in a Reifenberg sense) by zero sets of harmonic polynomials $p:\RR^n\rightarrow\RR$ of degree at most $d_0$.
\item The boundary $\partial\Omega$ can be partitioned into disjoint sets $\Gamma_d$ ($1\leq d\leq d_0$), where $x\in \Gamma_d$ if and only if tangent sets (geometric blowup) of $\partial\Omega$ at $x$ are zero sets of homogeneous harmonic polynomials $q:\RR^n\rightarrow\RR$ of degree $d$.
\item The ``regular set'' $\Gamma_1$ is relatively open and dense in $\partial\Omega$.
\item The boundary $\partial\Omega$ has upper Minkowski dimension $n-1$, while the ``singular set'' $\partial\Omega\setminus \Gamma_1=\Gamma_2\cup\dots\cup\Gamma_{d_0}$ has upper Minkowski dimension at most $n-3$.\end{enumerate}
\end{introthm}

For definitions of the terminology in Theorem \ref{intro:b}, see \S\ref{s: notation}. Theorem \ref{intro:b} is an amalgamation of several separate results: (i) was proved by Kenig and Toro in \cite{kenigtorotwophase}; (ii) and (iii) were proved by Badger in \cite{badgerharmonicmeasure} and \cite{badgerflatpoints}, respectively; and (iv) was proved by Badger, Engelstein, and Toro in \cite{BETHarmonicpoly} (also see \cite{localsetapproximation}). In fact, we proved in \cite{BETHarmonicpoly} that (ii) and (iii) hold on any closed set satisfying the bilateral approximation in (i). We also determined sharp estimates on the singular set in that scenario.  For a partial extension of Theorem \ref{intro:b} to a class of elliptic measures associated with variable coefficient operators, see Azzam and Mourgoglou \cite{am-elliptic}.
We reiterate that one distinction between the conclusion of Theorem \ref{intro:a} and Theorem \ref{intro:b} is that the former describes the global geometry of the free boundary up to a set of harmonic measure zero, while the latter describes the asymptotic geometry of the free boundary at every point.

The structural information and dimension estimates provided by Theorem \ref{intro:b} leave open the question of regularity of the free boundary when $\log f\in\VMO(d\omega^+)$ or $\log f\in C(\partial\Omega)$. On the other hand, regularity of $\Gamma_1$ has been addressed under a strengthened hypothesis:

\begin{introthm}[see Engelstein \cite{engelsteintwophase}] \label{intro:c} In addition to the hypothesis of Theorem \ref{intro:b}, assume that $\log h\in C^{l,\alpha}(\partial\Omega)$ for some $l\geq 0$ and $\alpha>0$ (resp.~$\log h\in C^\infty$, $\log h$ real analytic). Then the set $\Gamma_1$ is a $C^{l+1,\alpha}$ (resp.~ $C^\infty$, real analytic) $(n-1)$-dimensional manifold. \end{introthm}

Theorem \ref{intro:c} demonstrates that higher-order regularity of the free boundary data ensures higher-order regularity of the regular set in free boundary.
The goal of the present paper is to extend the conclusion of Theorem \ref{intro:c} to the singular set $\partial\Omega\setminus\Gamma_1$. We establish uniqueness of blowups and pseudo-blowups, and regularity of the singular set in the H\"older continuous regime:

\begin{theorem}[uniqueness of blowups] \label{t:unique} In addition to the hypothesis of Theorem \ref{intro:b}, assume that $\log h\in C^{0,\alpha}(\partial\Omega)$ for some $\alpha>0$. For all $1\leq d\leq d_0$ and $x\in\Gamma_d$, there exists a homogeneous harmonic polynomial $q:\RR^n\rightarrow\RR$ of degree $d$ such that for all sequences $x_i\in\Gamma_d$ with $x_i\rightarrow x$ and sequences of scales $r_i\downarrow 0$, $$\lim_{i\rightarrow\infty}\frac{\partial\Omega-x_i}{r_i}=q^{-1}(0)\quad\text{in the Attouch-Wets topology.}$$ In particular, $\partial\Omega$ has a unique tangent set at every $x\in\partial\Omega$.\end{theorem}

\begin{theorem}[regularity] \label{t:main} In addition to the hypothesis of Theorem \ref{intro:b}, assume that $\log h\in C^{0,\alpha}(\partial\Omega)$ for some $\alpha>0$. For every $\beta\in(0,\alpha/2)$, the ``singular set" $\partial\Omega\setminus\Gamma_1=\Gamma_2\cup\dots\cup\Gamma_{d_0}$ is contained in a countable union of $C^{1,\beta}$ manifolds of dimension at most $n-3$. \end{theorem}

In addition, we prove that the polynomial blowups in the singular set vary continuously with locally uniform H\"older modulus of continuity; see Corollary \ref{c:blowupschangeholder}.

A key ingredient in the proof of Theorem \ref{t:unique} that we wish to highlight is an epiperimetric inequality for harmonic functions (see Proposition \ref{epiperimetricinequality}). Epiperimetric inequalities were developed by Reifenberg \cite{ReifenbergEpi1}, \cite{ReifenbergEpi2} to establish analyticity of minimal surfaces at almost all points. These inequalities measure how ``isolated" certain critical points are in the space of all homogenous solutions, and lead to a convergence rate for blowups through ``improvement of flatness" or ``$\varepsilon$-regularity" type results. Taylor \cite{TaylorEpi} built on Reifenberg's approach and proved an epiperimetric inequality for area minimizers at singular points, leading to precise description of structure and size of the singular set for area minimizers. For free boundary problems, epiperimetric inequalities at flat points and along the top singular strata in the classical obstacle problem were proved by Weiss \cite{weissepiperimetric}, and in the thin obstacle problem proved independently  by Focardi and Spadaro \cite{FocardiSpadaro} and Garofalo, Petrosyan, and Smit Vega Garcia \cite{gpsvgepiperimetric}.
More recently,  Colombo, Spolaor, and Velichkov \cite{CSVEpi}, \cite{CoSpVe1} gave a constructive approach to prove log-epiperimetric inequalities at all singular points in the classic and thin obstacle problems. For further work in this direction, see Engelstein, Spolaor, and Velichkov \cite{esv1,esv2}. In each of these cases, the free boundary is variational in the sense that it can be represented as the zero set or graph of some function which minimizes an energy.  Our application of an epiperimetric inequality is novel in that we apply it to functions that do not minimize any energy, and thus, it is not clear how an epiperimetric inequality should translate into improved regularity. To wit, even for critical points of an energy functional, it is not clear, and may not be true, that an epiperimetric inequality implies a power rate of blowup.

\begin{remark}\label{r:sharpregularity}
There is an interesting question as to whether or not we can take $\beta = \alpha$ in Theorem \ref{t:main}, as we do in Theorem \ref{intro:c}. The loss of exponent enters the proof in \S\ref{s:boundW} when we apply the epiperimetric inequality to bound the growth of the Weiss type functional and it seems it cannot be avoided; this is not so surprising, as the epiperimetric inequality gives at best $C^{1,1/2-\epsilon}$ regularity of harmonic functions for any $\epsilon > 0$. When dealing with the regular set, as in Theorem \ref{intro:c}, the usual  technique to establish sharp regularity is to use the Hodograph transform and then apply bootstrapping arguments with (weighted) Schauder estimates. Our singular set is of codimension greater than two, so elliptic PDEs in the ambient space do not ``see" the set.
Thus, standard arguments do not apply and a new idea might be needed to address this question.
\end{remark}

\begin{remark}In addition to the hypothesis of Theorem \ref{intro:b}, assume that $\log h\in C^{0,\alpha}(\partial\Omega)$ for some $\alpha>0$. In subsequent work, McCurdy \cite{mccurdy} establishes Minkowski content and Hausdorff measure bounds on the singular set $\partial\Omega\setminus\Gamma_1$ and on the critical sets $S^{\pm} := \{x\in \Omega^{\pm} :|\nabla u^{\pm}(x)| = 0\}$ of the Green's functions of $\Omega^\pm$. McCurdy's result uses the new developments in quantitative stratification by Naber and Valtorta (see e.g.~ \cite{rectifiable-Reifenberg}). \end{remark}

\subsection{Weiss-type monotonicity formula, epiperimetric inequality, and strategy of the proof}\label{s: WeissMonotonicity}

We now recall the Weiss-type monotonicity formula for harmonic functions introduced by Garofalo and Petrosyan \cite{garofaloandpetrosyan} and review some of its basic properties. In doing so, it will become clear what quantities we need to estimate in order to prove our main result. This background might be well known to experts in free boundary problems, but may be less familiar to researchers working on harmonic measure problems.

For any $Q \in \RR^n$, $f\in W_{loc}^{1,2}(\RR^n)$, $r > 0$ and $d\in(0,\infty)$, define

\begin{equation}\label{eqn:weissmono} W_d(r, Q, f) := \frac{1}{r^{n-2+2d}}\int_{B(Q,r)} |\nabla f|^2 dx -\frac{d}{r^{n-1+2d}}\int_{\partial B(Q,r)} f^2d\sigma.\end{equation}
If $f \in C^1(\mathbb R^n)$, then differentiating $W_d$ in $r$ yields
\begin{equation}\begin{split}\label{formulaforwprime}
\frac{d}{dr} W_d(r,Q,f) =&\,\frac{n+2d-2}{r}\left(W_d(r,Q,\overline{f}_{r,Q})- W_d(r, Q, f)\right)\\ &\hbox{\ }\qquad\qquad\qquad + \frac{1}{r^{n-2+2d}}\int_{\partial B_r} \left(\nabla f\cdot \nu - \frac{d}{r}f\right)^2\,d\sigma,
\end{split}\end{equation}
where
\begin{equation}\label{newscaling}\overline{f}_{r,Q}(x) \colonequals \left(\frac{|x|}{r}\right)^df\left(r\frac{x}{|x|} + Q\right)\end{equation}
denotes the $d$-homogenous extension of $f|_{\partial B(Q,r)}$. For a detailed derivation in the case $d = 3/2$ (and with variable coefficients), the reader may consult \cite[Theorem 4.3]{gpsvgepiperimetric}.

\begin{observation}\label{controlthegrowth} Assume that $f:\RR^n\rightarrow\RR$ is harmonic. Then $f$ minimizes the Dirichlet energy in its trace class. Hence $W_d(r, Q, f) \leq W_d(r, Q, \overline{f}_{r,Q})$ and \begin{equation}\label{derivativeofWcontrolsrateofblowup}\frac{d}{dr} W_d(r,Q, f) \geq \frac{1}{r^{n-2+2d}}\int_{\partial B_r} \left(\nabla f\cdot \nu - d\frac{f}{r}\right)^2d\sigma.\end{equation} That is to say, the growth of $W_d(r, Q,f)$ controls how far away $f$ is from its $d$-homogenous extension (as $\nabla f\cdot \nu - (d/r)f \equiv 0$ if and only if $f$ is $d$-homogenous). Moreover, it follows that $W_d(r,Q,f)$ is monotone in $r$, and hence, $W_d(0, Q, f)=\lim_{r\downarrow 0} W_d(0,Q,f)$ exists. Furthermore, if $f$ vanishes to order at least $d$ at $Q$, then $W_d(r, Q, f) \geq 0$ with equality if and only if $f$ is $d$-homogenous in $B(Q,r)$.
\end{observation}

To estimate the growth of $W_d(r,Q,f)$ from above, we need an epiperimetric inequality for $d$-homogenous harmonic polynomials. Proposition \ref{epiperimetricinequality}, whose proof we defer to the appendix, is a simple consequence of the eigenvalues of the spherical Laplacian.

\begin{proposition}[an epiperimetric inequality for harmonic functions]\label{epiperimetricinequality}
For every integer $n\geq 2$ and real number $d>0$, there exists $\kappa\in(0,1)$ such that if $u \in W^{1,2}(B(Q,r))$ is homogeneous of degree $d$ about $Q$ and $f$ denotes the harmonic extension of $u|_{\partial B(Q,r)}$ to $B(Q,r)$, then
\begin{equation}\label{eq:epiperimetricinequality} W_d(r,Q,f) \leq (1-\kappa)W_d(r,Q,u).\end{equation}
In fact, when $d$ is an integer we can take $\kappa = 1/(n + 2d-1)$.
\end{proposition}

Using the epiperimetric inequality \eqref{eq:epiperimetricinequality}, one can show that for all harmonic functions $f$ that vanish to order at least $d$ at $Q$ and for all scales $0<t/2<s<t\leq 1$,
\begin{equation}\label{eqn-tt1}
\left(\int_{\partial B_1(0)}\left| \frac{f(tx + Q)}{t^d} - \frac{f(sx+Q)}{s^d}\right|^2\ d\sigma\right)^{1/2} \leq t^{\gamma/2}\sqrt{W(1, Q, f)},\quad
\end{equation} where  $\gamma :=(n+2d-2)\kappa/(1-\kappa)$. Thus, the blowups  $r^{-d}f(rx+Q)$ of $f$ converge at a power rate as $r \downarrow 0$ to a unique function $f^{(Q)}$, uniformly across $Q$ in any compact set, and the functions $f^{(Q)}$ vary H\"older continuously in $Q$. The arguments in \S\ref{s: c1betaconvergence} yield \eqref{eqn-tt1} as a special case.

Our basic strategy for the two-phase problem with H\"older data is to replicate the argument sketched above. For each $Q\in \partial \Omega$, we consider the \emph{jump function} \begin{equation}\label{e:defofvq} v^{(Q)}(x) := h(Q)u^+(x) -  u^-(x),\end{equation} where $u^\pm$ are Green's functions of $\Omega^{\pm}$ (see \S\ref{s: notation}). We want to study the behavior of $W_d(r, Q, v^{(Q)})$ as $r\downarrow 0$. Unfortunately, $v^{(Q)}$ is not harmonic unless $\omega^+=\omega^-$. Therefore, unlike the situation above, one cannot expect $r \mapsto W_d(r, Q, v^{(Q)})$ to be monotone in $r$. Nevertheless, we observe that \begin{equation}\label{distributionderivative}
\Delta v^{(Q)}(x) = (h(Q)d\omega^+-d\omega^-)|_{\partial \Omega} = \left(\frac{h(Q)}{h(x)} - 1\right)d\omega^-|_{\partial \Omega}
\end{equation} in the sense of distributions. Thus, $\log(h) \in C^{0,\alpha}(\partial \Omega)$ implies \begin{equation}\label{e:boundondistderivative}|\Delta v^{(Q)}(x)| \leq C|x-Q|^\alpha d\omega^-|_{\partial \Omega}.\end{equation} That is, $v^{(Q)}$ is ``almost harmonic", and $W_d(r,Q,v^{(Q)})$ is ``almost monotone".
The estimate \eqref{e:boundondistderivative} allows us to prove that $v^{(Q)}$ ``almost minimizes" the Dirichlet energy in balls centered at $Q\in \partial \Omega$. In conjunction with Proposition \ref{epiperimetricinequality}, we are able to conclude H\"older continuity of the blowup in essentially the way outline above modulo the issue of degeneracy (the ``almosts" create extra, mostly technical, difficulties).

The issue of degeneracy occurs when we need to control the error caused by the fact that $\Delta v^{(Q)}(B(Q,r)) \neq 0$ (where the Laplacian can be evaluated formally and treated as a measure). The most problematic terms are proportional to \begin{equation}\label{e:densityquotient} \frac{\omega^\pm(B(Q,r))}{r^{n-2+d}}\quad\text{for }Q \in \Gamma_d, r > 0.\end{equation} The major difficulty is ruling out the possibility that this quotient tends to zero or infinity as $r\downarrow 0$ (this is the degeneracy alluded to above). Using structural results from \cite{BETHarmonicpoly} and an argument inspired by \cite{kenigtoroduke}, we show that while \emph{a priori} the quotient in \eqref{e:densityquotient} may blowup or degenerate as $r\downarrow 0$, it cannot do so faster than any power of $r$; for a precise statement, see Lemma \ref{growthofomega}. Furthermore, this result holds uniformly over compact subsets of $\Gamma_d$. This is sufficient control to establish almost-monotonicity and bound the growth of $W_d(r, Q, v^{(Q)})$. Once we have these bounds, we can obtain  the existence, positivity, and finiteness of $\lim_{r\downarrow 0} r^{-(n-2+d)}\omega^-(B(Q,r))$ using an argument inspired by \cite{gpsvgepiperimetric}; see Theorem \ref{l:blowupsunique}. That this limit exists is an essential step in gaining geometric information from the blowup process outlined in \S\ref{s: notation}.

\subsection{Plan of the paper}  The rest of the paper naturally splits into three parts.

In the first part, \S\S \ref{s: notation} and \ref{sec:start}, we introduce essential notation and ideas from geometric measure theory and the theory of harmonic measure. In particular, we recall Kenig and Toro's blowup analysis for the two-phase problem for harmonic measure \cite{kenigtorotwophase}. Combined with results of \cite{BETHarmonicpoly}, this allows us to prove that the quotient in \eqref{e:densityquotient} does not go to zero or infinity faster than any power of $r$; see Lemma \ref{growthofomega}. The start of the proof of the main theorems is given in \S\ref{sec:start}.

In the second part, \S\S \ref{formulasandcomputations} and \ref{s:boundW}, we control the growth of the Weiss monotonicity formula applied to the function $v^{(Q)}$ (see \eqref{e:defofvq}). Because of the possibility of degeneracy, we must actually first work with the Almgren frequency formula. This effort culminates in a H\"older growth rate of the function $r\mapsto W_d(r, Q, v^{(Q)})$; see Proposition \ref{p: Wdoesntgrowtoofast}.

In the third and final part, \S\S \ref{s: c1betaconvergence} and \ref{s: rectifiability}, we translate our bounds on the growth of $W_d(r, Q, v^{(Q)})$ to geometric information about the singular set. In \S\ref{s: c1betaconvergence}, we show uniqueness of blowups (Theorem \ref{l:blowupsunique}) and their H\"older continuous dependence on the point $Q$. Then, in \S\ref{s: rectifiability}, we obtain regularity of the singular set using a classical argument based on the implicit function theorem.

\vspace{-.2 cm}

\subsection*{Acknowledgements} The authors would like to thank the two anonymous referees whose careful reading and thoughtful comments led to improvements in the manuscript. M.~Engelstein would like to thank L.~Spolaor and B.~Velichkov for many enlightening conversations about epiperimetric inequalities. Part of this research and work on this manuscript was carried out while the three authors attended the long programs on Harmonic Analysis at MSRI in Spring 2017 and at PCMI in Summer 2018.

\section{Blowups of harmonic measure on NTA domains}\label{s: notation}

In this section, we first review terminology appearing in Theorem \ref{intro:b}, including Jerison and Kenig's class of non-tangentially accessible domains, Kenig and Toro's blowup analysis of harmonic measure, and Badger and Lewis' framework for local set approximation. For full details, readers are referred to \cite{jerisonandkenig}, \cite{kenigtoroannals}, \cite{kenigtorotwophase}, \cite{localsetapproximation}, and the references therein. Towards the end of this section, we use results from our previous work \cite{badgerharmonicmeasure}, \cite{BETHarmonicpoly} to compute the local dimension of harmonic measures $\omega^\pm$ at points $x\in\Gamma_d$, where blowups of the boundary $\partial\Omega^\pm$ are zero sets of homogeneous harmonic polynomials of degree $d$; see Lemma \ref{growthofballsatsingularpoints}

\begin{definition}[\cite{jerisonandkenig}] \label{ntadomains}
A domain (i.e.~a connected, open set) $\Omega \subset \mathbb R^n$ is called \emph{NTA} or \emph{non-tangentially accessible} if there exist constants $M_\Omega>1$ and $R_\Omega > 0$ such that the following hold:
\begin{enumerate}
\item $\Omega$ satisfies the \emph{corkscrew condition}: for all $Q \in \partial \Omega$ and $0 < r < R_\Omega$, there exists $x\in \Omega\cap B(Q,r)$ such that $\dist(x,\partial\Omega)>M_\Omega^{-1}r$.
\item $\RR^n\setminus\Omega$ satisfies the corkscrew condition.
\item $\Omega$ satisfies the \emph{Harnack chain condition}: If $x_1, x_2\in \Omega\cap B(Q,r/4)$ for some $Q\in \partial \Omega$ and $0<r<R_\Omega$, and $\dist(x_1,\partial\Omega)>\delta$, $\dist(x_2,\partial\Omega)>\delta$, and $|x_1-x_2|<2^l\delta$ for some $\delta>0$ and $l\geq 1$, then there exists a chain of no more than $Ml$ overlapping balls  connecting $x_1$ to $x_2$ in $\Omega$ such that for each ball $B=B(x,s)$ in the chain: \begin{align*}&M_\Omega^{-1}s<\gap(B,\partial\Omega)<M_\Omega s, &&\gap(B,\partial\Omega)=\inf_{x\in B}\inf_{y\in\partial\Omega}|x-y|, \\ \diam B> &M_\Omega^{-1}\min\{\dist(x_1,\partial\Omega),\dist(x_2,\partial\Omega)\}, &&\diam B=\sup_{x,y\in B}|x-y|.\end{align*}
\end{enumerate} We refer to $M_\Omega$ and $R_\Omega$ as \emph{NTA constants} of the domain $\Omega$. When $\partial\Omega$ is unbounded, $R_\Omega=\infty$ is allowed. To distinguish between (i) and (ii), the former may be called the \emph{interior corkscrew condition} and the latter may be called the \emph{exterior corkscrew condition}.
\end{definition}

The exterior corkscrew condition guarantees that NTA domains are regular for the classical Dirichlet problem for harmonic functions (i.e.~ continuous solutions exist for all continuous boundary data), and therefore, the harmonic measures of an NTA domain exist by the Perron-Brelot-Weiner method and Riesz representation theorem. For background on the relevant potential theory, see e.g. Helms \cite{helms}.

\begin{proposition} Let $\Omega\subset\RR^n$ be a Wiener regular domain (such as an NTA domain). When $\Omega$ is bounded, there exists a unique family of Borel regular probability measures $\{\omega^X\}_{X\in\Omega}$ on $\partial\Omega$ such that $$X\mapsto \int_{\partial\Omega} f\,d\omega^X$$ solves the Dirichlet problem with boundary data $f\in C(\partial\Omega)$. When $\Omega$ is unbounded, there exists a unique family of Borel regular probability measures $\{\omega^X\}$ on $\partial_\infty\Omega=\partial\Omega\cup\{\infty\}$ such that $$X\mapsto \int_{\partial_\infty\Omega} f\,d\omega^X$$ solves the Dirichlet problem with boundary data $f\in C(\partial_\infty\Omega)$. The measure $\omega^X$ is called the \emph{harmonic measure of $\Omega$ with pole at $X$}.\end{proposition}

The next lemma states that harmonic measures on NTA domains are locally doubling, a form of weak regularity of a measure.

\begin{lemma}[{\cite[Lemmas 4.9 and 4.11]{jerisonandkenig}}] Let $\Omega\subset\RR^n$ be an NTA domain, and let $K\subset\partial\Omega$ be a compact set. There exists $C>1$ depending only on the NTA constants of $\Omega$ and on $K$ such that for all $Q\in K$, $0<2r<R_\Omega$, and $X\in\Omega\setminus B(Q,2M_\Omega r)$, $$\omega^X(\partial\Omega\cap B(Q,2s))\leq C \omega^X(\partial\Omega\cap B(Q,s))\quad\text{for all }0<s\leq r.$$\end{lemma}

On unbounded NTA domains, there is a related notion of harmonic measure with pole at infinity.

\begin{proposition}[{\cite[Lemma 3.1, Corollary 3.2]{kenigtoroannals}}] \label{prop:infinity} Let $\Omega\subset\RR^n$ be an unbounded NTA domain with $R_\Omega=\infty$. There exists a function $v$ such that $\Delta v=0$ in $\Omega$, $v>0$ in $\Omega$, and $v=0$ on $\partial\Omega$ that is unique up to scaling by a positive constant.  For each $Q\in\partial\Omega$, there exists a unique doubling Radon measure $\omega^\infty$ on $\partial\Omega$ such that $$\int_{\partial\Omega} \varphi\, d\omega^\infty=\int_\Omega u\Delta \varphi\quad\text{for all }\varphi\in C_c^\infty(\RR^{n}),$$ where $u$ satisfies $\Delta u=0$ in $\Omega$, $u>0$ in $\Omega$, and $u=0$ on $\partial\Omega$, and $\omega^\infty(\partial\Omega\cap B(Q,1))=1.$\end{proposition}

Below we use a standard convention and call $\omega$ the \emph{harmonic measure of $\Omega$} if $\omega=\omega^X$ is a harmonic measure with pole at $X\in\Omega$ or if $\omega=\omega^\infty$ is a harmonic measure of $\Omega$ with pole at infinity. In each occurrence, the pole of the measure is fixed.

\begin{definition}[{\cite[Definitions 4.2 and 4.3]{kenigtorotwophase}}] Let $\Omega\subset\RR^n$ be an NTA domain with harmonic measure $\omega$. We say that  $f \in L^2_{\mathrm{loc}}(d\omega)$ belongs to $\BMO(d\omega)$ if \begin{equation*}\label{bmodef}
\sup_{r > 0} \sup_{Q\in \partial \Omega} \left(\fint_{B(Q,r)} |f-f_{Q,r}|^2\, d\omega\right)^{1/2} < \infty,
\end{equation*}
where $f_{Q,r} = \fint_{B(Q,r)}f \, d\omega$ denotes the average of $f$ over the ball. We denote by $\VMO(d\omega)$ the closure in $\BMO(d\omega)$ of the set of uniformly continuous bounded functions on $\partial\Omega$.
\end{definition}

Suppose that $\Omega^+=\Omega$ and $\Omega^-=\RR^n\setminus\overline{\Omega}$ are complimentary NTA domains with harmonic measures $\omega^+$ and $\omega^-$, respectively. If $\Omega^+$ and $\Omega^-$ are both unbounded and $R_{\Omega^\pm}=\infty$, then we require either that $\omega^+$ and $\omega^-$ have poles $X^\pm\in\Omega^\pm$ or that $\omega^+$ and $\omega^-$ both have poles at infinity. Otherwise, we assume $\omega^+$ and $\omega^-$ have finite poles. Following \cite{kenigtorotwophase}, we call $\Omega^+$ and $\Omega^-$ \emph{two-sided NTA domains}. When $\omega^\pm$ have poles at $X^\pm\in\Omega^\pm$, we let  $u^\pm$ denote the Green function of $\Omega^\pm$ with pole at $X^\pm$, i.e. the unique functions satisfying $$\int_{\partial \Omega} \varphi\, d\omega^{\pm}=\int_{\Omega^\pm} u^{\pm} \Delta \varphi\quad\text{for all }\varphi \in C^{\infty}_c(\mathbb R^n\setminus \{X^{\pm}\}).$$ When $\omega^\pm$ have poles at infinity, we let $u^\pm$ denote the functions given by Proposition \ref{prop:infinity}.

Let $Q,Q_j\in\partial\Omega$ with $Q_j\rightarrow Q$, and let $r_j>0$ with $r_j\downarrow 0$. Following \cite{kenigtorotwophase}, we define associated sequences $\Omega^\pm_j$, $\partial\Omega_j$, $u_j^\pm$, and $\omega^\pm_j$ by

\begin{equation}\begin{split}\label{blowup}
\Omega^\pm_j &:=\frac{\Omega^\pm-Q_j}{r_j},\qquad\qquad\qquad \partial\Omega_j :=\frac{\partial\Omega-Q_j}{r_j},\\  u^{\pm}_j(x) &:=\frac{u^{\pm}(r_jx + Q_j) r_j^{n-2}}{\omega^{\pm}(B(Q_j,r_j))}, \quad \omega^{\pm}_j(E) \colonequals  \frac{\omega^\pm(r_jE + Q_j)}{\omega^{\pm}(B(Q_j, r_j))}.
\end{split}\end{equation}

\begin{theorem}[{\cite[Theorem 4.2]{kenigtorotwophase}}]\label{t:blowups} With the assumptions and notation above, there exists a subsequence of $(\Omega^\pm_j,\partial\Omega_j,\omega^\pm_j,u^\pm_j)$, which we relabel, and there exist unbounded two-sided NTA domains $\Omega^\pm_\infty$, harmonic measures $\omega^\pm_\infty$ with pole at infinity, and Green functions $u^\pm_\infty$ with pole at infinity such that \begin{itemize}
\item $\overline{\Omega^\pm_j}\rightarrow \overline{\Omega^\pm_\infty}$ and $\partial\Omega_j\rightarrow \partial\Omega_\infty$ in the Attouch-Wets topology,
\item $\omega_j^\pm$ converges to $\omega_\infty^\pm$ in the vague topology (that is, weakly as Radon measures), and\item $u_j^\pm \rightarrow u^\pm_\infty$ uniformly on compact sets.\end{itemize} We call the tuple $(\Omega_\infty^\pm, \partial\Omega_\infty, \omega^\pm_\infty, u^\pm_\infty)$ a \emph{pseudoblowup of the harmonic measure} at $Q$. When $Q_j=Q$ for all $j$, a pseudoblowup is called a \emph{blowup}.
\end{theorem}

\begin{remark}Following \cite{kenigtoroannals}, the measures $\omega^\pm_\infty$ are called \emph{pseudotangent measures} of $\omega^\pm$ at $Q$. When $Q_j=Q$ for all $j$, the measures $\omega^\pm_\infty$ are called \emph{tangent measures} of $\omega^\pm$ at $Q$. Tangent measures of general Radon measures were first introduced by Preiss \cite{preiss}.\end{remark}

\begin{remark}[Local set approximation] A sequence of nonempty closed sets $F_j\subset\RR^n$ converges to a nonempty closed set $F\subset\RR^n$ in the \emph{Attouch-Wets topology} if for every $r>0$, $$\lim_{j\rightarrow\infty}\excess(F_j\cap B(0,r),F)=0\quad\text{and}\quad\lim_{j\rightarrow\infty} \excess(F\cap B(0,r),F_j)=0,$$ where $\excess(A,B)=\sup_{x\in A}\inf_{y\in B}|x-y|$ and $\excess(\emptyset,B)=0$ for all nonempty sets $A,B\subset\RR^n$. For general background, see Beer \cite{beer}. This topology is a convenient choice, because it is metrizable and for every nonempty compact set $K\subset\RR^n$, the collection $\mathfrak{C}(K)$ of closed sets in $\RR^n$ that intersect $K$ is sequentially compact.

Let $F\subset\RR^n$ be a nonempty closed set and let $x\in F$. Following \cite{localsetapproximation} (motivated by \cite{preiss}, \cite{kenigtoroannals}), we say that a nonempty closed set $T\subset\RR^n$ is a \emph{pseudotangent set of $F$ at $x$} if there exist sequences $x_j\in F$ with $x_j\rightarrow x$ and $r_j>0$ with $r_j\downarrow 0$ such that $$\frac{F-x_j}{r_j}\rightarrow T\quad\text{in the Attouch-Wets topology}.$$ If $x_j=x$ for all $j$, we call $T$ a \emph{tangent set} of $F$ at $x$. Let $\mathcal{S}$ be a collection of nonempty closed sets containing the origin such that $S\in \mathcal{S}$ and $\lambda>0$ implies $\lambda S\in\mathcal{S}$. We say that a nonempty set $E\subset\RR^n$ is \emph{locally bilaterally well approximated by $\mathcal{S}$} if all pseudotangent sets of $\overline{E}$ belong to $\overline{\mathcal{S}}$. Equivalently (see \cite[\S4]{localsetapproximation}), $E$ is locally bilaterally well approximated by $\mathcal{S}$ if and only if for all compact sets $K\subset E$ and $\varepsilon>0$, there exists $r_{K,\varepsilon}>0$ such that \begin{quotation}for locations $x\in K$ and scales $0<r\leq r_{K,\varepsilon}$, there exists $S\in\mathcal{S}$ such that $\excess(E\cap B(x,r),x+S)< \varepsilon r$ and $\excess((x+S)\cap B(x,r),E)<\varepsilon r$.\end{quotation} Thus, sets that are locally bilaterally well approximated by $\mathcal{S}$ are like \emph{Reifenberg sets with vanishing constants} (e.g.~see \cite{kenigtoroannals}), except with approximation by hyperplanes replaced by approximation by sets in $\mathcal{S}$. See \cite{localsetapproximation} for further discussion. Using this terminology,
Theorem \ref{t:blowups} says that the pseudotangent sets of the boundary of a two-sided NTA domain are boundaries of unbounded two-sided NTA domains.\end{remark}

The following theorem identifies pseudotangents of harmonic measure under VMO or continuous two-phase free boundary conditions.

\begin{theorem}[{\cite[Theorem 4.4]{kenigtorotwophase}}]\label{t:vmoblowups} Under the hypothesis of Theorem  \ref{intro:b} (in particular, under the assumption that $\log h\in\VMO(d\omega^+)$), every pseudoblowup $(\Omega^\pm_\infty,\partial\Omega_\infty,\omega^\pm_\infty, u^\pm_\infty)$ of the harmonic measure satisfies $\omega_\infty^+=\omega_\infty^-$ and $u_\infty:=u_\infty^+-u_\infty^-$ is a harmonic polynomial of degree at most $d_0$ depending on the NTA constants of $\Omega^\pm$.\end{theorem}

Together, Theorems \ref{t:blowups} and \ref{t:vmoblowups} ensure that the boundary $\partial\Omega$ is locally bilaterally well approximated by zero sets of harmonic polynomials, whose positive and negative sets are unbounded NTA domains. In this setting, refined information about the boundary was obtained in \cite{badgerharmonicmeasure}, \cite{badgerflatpoints}, and \cite{BETHarmonicpoly}. See Theorem \ref{intro:b} above. In addition to the consequences listed there, we can classify the pseudotangent sets of $\partial\Omega$ along sequences in $\Gamma_d$, where the tangent sets of $\partial\Omega$ are zero sets of homogeneous harmonic polynomials of degree $d$. This is a consequence of the special geometry of harmonic varieties; see \cite[Theorem 1.4]{BETHarmonicpoly}.

\begin{theorem}[{\cite[Theorem 1.1(iv)]{BETHarmonicpoly}}] \label{homogenouspseudoblowups}
With the same notation as in Theorem \ref{intro:b}, assume $Q_j \in \Gamma_d$ converges to $Q\in\Gamma_d$. For all $r_j \downarrow 0$, there is a subsequence of $(Q_j,r_j)$, which we relabel, such that the pseudotangent set $\lim_{j\rightarrow\infty}(\partial \Omega - Q_j)/r_j$ is the zero set of a homogeneous harmonic polynomial of degree $d$.
\end{theorem}

We now recall some facts about \emph{polynomial harmonic measures}.
Let $p:\RR^n\rightarrow\RR$ be a nonconstant harmonic polynomial. Then $\pm p$ is a Green function with pole at infinity for any connected component of $\{x:\pm p(x)>0\}$. Hence there exists (e.g.~see \cite{badgerharmonicmeasure}) a unique Radon measure $\omega_p$ on $\Sigma_p=\{x:p(x)=0\}$ such that \begin{equation}\label{greenfunctionatinfinity}\int_{\Sigma_p} \varphi\, d\omega_{p} = \int p^\pm \Delta \varphi,\quad\text{for all }\varphi\in C^\infty_c(\RR^n).\end{equation} As noted above, Kenig and Toro \cite{kenigtorotwophase} proved that when $\Omega^\pm$ are two-sided NTA domains, $\omega^+\ll\omega^-\ll\omega^+$ and $\log(d\omega^-/d\omega^+)\in\VMO(d\omega^+)$, all pseudoblowups of the harmonic measures $\omega^\pm$ are polynomial harmonic measures $\omega_p$, where $\{x:\pm p(x)>0\}$ are unbounded NTA domains. However, we note that polynomial harmonic measures can be defined even if the components of $\{x:\pm p(x)>0\}$ are not NTA (an example of such a polynomial is $p(x,y,z)=x^2-y^2+z^3-3x^2z$, see \cite[Example 5.1]{lm15}).

\begin{lemma}[{\cite[Lemma 4.2]{badgerharmonicmeasure}}] \label{l:hball} If $p:\RR^n\rightarrow\RR$ is a homogeneous harmonic polynomial of degree $d\geq 1$, then \begin{equation}\omega_p(B(0,r))= \frac{d}{2} r^{n-2+d}\|p\|_{L^1(S^{n-1})}\quad\text{for all }r>0,\end{equation} \begin{equation}\frac{\omega_p(B(0,\tau r))}{\omega_p(B(0,r))}=\tau^{n-2+d}\quad\text{for all }\tau,r>0.\end{equation}\end{lemma}

\begin{remark}\label{r:ball-norm} By our choice of normalization (see \eqref{blowup}), if $p=u_\infty^+-u_\infty^-$ is a pseudoblowup appearing when $\log (d\omega^-/d\omega^+)\in\VMO(d\omega^+)$, then $\omega_p(B(0,1))=1$. Thus, if $\lambda p$ and $\lambda' p$ both appear as pseudoblowups (even at different base points), then $\lambda = \lambda'$. Furthermore, since $\{\pm p > 0\}$ are NTA domains with NTA constants determined the NTA constants of the original domain $\Omega$,  standard estimates for the Green's function on NTA domains (e.g.~see \cite[Lemma 4.8]{jerisonandkenig}) ensure that \begin{equation}\label{e:sizeofp}\sup_{x\in B(0,1)} |p(x)| \simeq \sup_{x\in B(0,1)} u_\infty^\pm(x)\simeq \omega_p(B(0,1)) =1,\end{equation} where the implicit constants depend only on the NTA constants of $\Omega$. It follows that for every compact set $K\subset\partial\Omega$, $$\sup_{Q\in K}\sup_{0<r\leq 1} \frac{u^\pm(rx+Q) r^{n-2}}{\omega^-(B(Q,r))} \simeq 1$$ for all $0<r\leq r_0(K)$. \end{remark}

At flat points $Q\in\Gamma_1$, it is known that $\omega$ is locally asymptotically optimally doubling (see \cite{kenigtoroduke}, \cite{DKT}, \cite{badgerflatpoints}). In particular, when $Q\in\Gamma_1$, $\omega^{\pm}(B(Q,r))/r^{n-1}$ does not grow faster or slower than any power of $r$. We now verify that this behavior persists at singular points $Q\in\Gamma_d$, $d>1$. In the parlance of geometric measure theory, the following lemma implies the \emph{local dimension} (e.g.~see \cite[p. 25]{falconer}) of $\omega^\pm$ at $Q\in\Gamma_d$ is $n-2+d$.

\begin{lemma}\label{growthofballsatsingularpoints} With the same notation as in Theorem \ref{intro:b}, let $1\leq d\leq d_0$ and assume that $K$ is a compact subset of $\Gamma_d$. Then, for every $\delta > 0$, \begin{equation}\label{growthofomega} \liminf_{r\downarrow 0}\inf_{Q\in K} \frac{\omega^{\pm}(B(Q,r))}{r^{n-2+d+\delta}} = \infty,\qquad \limsup_{r\downarrow 0}\sup_{Q\in K} \frac{\omega^{\pm}(B(Q,r))}{r^{n-2+d-\delta}} = 0.\end{equation}
\end{lemma}

\begin{proof} Fix a compact set $K\subset\Gamma_d$. We first prove that \begin{equation}\label{asod} \lim_{r\downarrow 0}\sup_{Q\in K} \left|\frac{\omega^\pm(B(Q,\tau r))}{\omega^\pm(B(Q,r))} -\tau^{n-2+d}\right| = 0\quad\text{for all }\tau>0.\end{equation} By Theorem \ref{homogenouspseudoblowups}, for every sequence $Q_j\in K$ with $Q_j\rightarrow Q$ and for every sequence $r_j \downarrow 0$, there is a subsequence (which we relabel) and a $d$-homogenous harmonic polynomial $p$, such that $\omega^\pm_j \rightharpoonup \omega_{p}$. By Lemma \ref{l:hball} and Remark \ref{r:ball-norm}, $\omega_p(B(0,\tau r)) = \tau^{n-2+d}\omega_p(B(0,r))$ and $\omega_p(B(0,1))=1$. It follows that $$\omega_{p}(B(0,\tau)) = \tau^{n-2+d}\quad\text{and}\quad\omega_p(\partial B(0,\tau))=0\quad\text{for all }\tau>0.$$  Therefore, recalling the definition of $\omega^\pm_j$ \begin{equation}\label{showingasod}\lim_{j\rightarrow \infty} \frac{\omega^\pm(B(Q_j,\tau r_j))}{\omega^\pm(B(Q_j,r_j))} = \lim_{j\rightarrow \infty} \omega^\pm_j(B(0,\tau)) = \omega_{p}(B(0,\tau)) = \tau^{n-2+d}
\end{equation} for all $\tau>0$. Because \eqref{showingasod} holds for an arbitrary initial sequence, we obtain \eqref{asod}.

To proceed, fix $\delta > 0$ and $\tau \in(0,1)$. Pick $\varepsilon>0$ to be specified later. By \eqref{asod}, there exists $r_0 > 0$ such that for all $Q\in K$ and $0<r\leq r_0$, \begin{equation}\label{sandwichomega} (1-\varepsilon)\tau^{n-2+d}\leq  \frac{\omega(B(Q,\tau r))}{\omega(B(Q,r))} \leq (1+\varepsilon)\tau^{n-2+d}\end{equation} Fix $Q\in K$. Iterating \eqref{sandwichomega} $k$ times, we obtain \begin{equation}\label{iteratesandwichomega} (1-\varepsilon)^k \tau^{k(n-2+d)} \omega(B(Q,r_0)) \leq \omega(B(Q,\tau^k r_0)).\end{equation} Dividing both sides by $(\tau^k r_0)^{n-2+d+\delta}$ yields  $$r_0^{-(n-2+d+\delta)} \left(\frac{1-\varepsilon}{\tau^\delta}\right)^k \leq \frac{\omega(B(Q,\tau^kr_0))}{(\tau^kr_0)^{n-2+d+\delta}}.$$ Thus, provided we chose $\varepsilon$ sufficiently small to guarantee $(1-\varepsilon)/\tau^d>1$, $$\liminf_{k\rightarrow\infty} \frac{\omega^\pm(B(Q,\tau^kr_0))}{(\tau^kr_0)^{n-2+d+\delta}}=\infty\quad\text{(uniformly over $Q\in K$)}.$$ The first inequality in \eqref{growthofomega} now follows by a standard argument. The proof of the second inequality is very similar.
\end{proof}

We end this section with an observation on the convergence of the gradients under these pseudo-blowups.

\begin{remark}\label{r:weakstarconvergence}
In  \cite[Proposition 4.4]{engelsteintwophase}, the second author proved under the assumption $\log(h) \in C(\partial \Omega)$ that the Green's functions, $u^{\pm}$, are locally Lipschitz. This Lipschitz bound depended only on the $L^2$ norm of $u^{\pm}$ in the compact set and the $L^\infty$ norm of $\log(h)$. In fact note that the $u_j^{\pm}$ are uniformly locally bounded in $L^2_{\mathrm{loc}}$ (as the sequence is uniformly bounded in $L^\infty_{\mathrm{loc}}$) and that $h_j(P) = \frac{d\omega_j^{-}}{d\omega_j^+}(P) = h(r_jP +Q_j)$. Thus recalling that the $u_j^{\pm}$ are harmonic on $\{u_j^{\pm}>0\}$ and repeating the arguments in Section 4 of \cite{engelsteintwophase} we conclude that the functions $u_j^{\pm}$ are uniformly locally Lipschitz.

Also note that the uniform convergence of $u^\pm_j$ to $p^{\pm}$ implies convergence in $C^\infty$ in compact subsets of $\{p^{\pm} > 0\}$ (as the $u_j^\pm$ are harmonic there). To show that the $|\nabla u^{\pm}_j|$ converges to $|\nabla p^{\pm}|$ both in $L^2_{\mathrm{loc}}$ and weak-star in $L^\infty$ (the former implies the latter), recall that the $p^{\pm}$ are locally Lipschitz and $|\{x\in B_R(0) : \mathrm{dist}(x, \{p^{\pm} = 0\}) \leq \varepsilon\}| < C_R \varepsilon$  (see, e.g. \cite{cheegernabervaltorta}, though for harmonic polynomials it is a simple consequence of analyticity).  Thus, by the prior considerations and the local uniform Lipschitz character of $u_j^{\pm}$,
$$\limsup_{j\to\infty} \int_{|\{x\in B_R(0) : \mathrm{dist}(x, \{p^{\pm} = 0\}) \leq \varepsilon\}|} |\nabla u_j^{\pm}|^2 + |\nabla p^{\pm}|^2 \leq C_R\varepsilon.
$$
The desired $L^2_{\mathrm{loc}}$ convergence follows from the triangle inequality and the $C^\infty$ convergence of $u_j^{\pm}$ to $p^\pm$ in compact subsets of $\{p^{\pm} > 0\}$.
\end{remark}

\section{Start of the proof of the main theorems}\label{sec:start}

For the remainder of the paper, we let $\Omega^+=\Omega\subset\RR^n$ and $\Omega^-=\RR^n\setminus\overline{\Omega}$ denote a fixed pair of two-sided NTA domains. We assume that $\Omega^\pm$ are unbounded with NTA constants $R_\Omega^\pm=\infty$ and $M_\Omega^\pm$. Furthermore, we assume that $\omega^\pm$ are harmonic measures for $\Omega^\pm$ with pole at infinity such that $\omega^+\ll\omega^-\ll\omega^+$ and the Radon-Nikodym derivative $h=d\omega^-/d\omega^+$ satisfies $\log h\in C^{0,\alpha}(\partial\Omega)$ for some $\alpha>0$. Thus, we are in a regime where the conclusions of Theorems \ref{intro:b} and \ref{intro:c} hold. We adopt all notations and conventions set in \S\ref{s: notation}.

\begin{remark} We will be content to prove Theorems \ref{t:unique} and \ref{t:main} under the assumptions above. When $\omega^\pm$ have finite poles, the proofs go through with only minor modifications: one must restrict estimates to scales where the poles $X^\pm$ are relatively far away from the relevant portion of the boundary. A diligent reader with enough paper will have no difficulty verifying the details.\end{remark}

For each $Q\in \partial\Omega$, we define, as in \cite{engelsteintwophase}, the \emph{jump function} $v^{(Q)}:\RR^n\rightarrow\RR$ by \begin{equation}v^{(Q)}(x) = h(Q)u^+(x)-u^-(x)\quad\text{for all }x\in\RR^n,\end{equation} where $u^\pm$ are Green's functions with pole at infinity corresponding to $\omega^\pm$, extended beyond $\overline{\Omega^\pm}$ by setting $u^\pm\equiv 0$ on $\Omega^\mp$. By definition of the Green's functions (see Proposition \ref{prop:infinity}), we have $\Delta u^\pm=0$ in $\Omega^\pm$ and $u^\pm=0$ on $\partial\Omega$. There is no reason to expect that $v^{(Q)}$ defines a global harmonic function (in fact, it does not unless $\omega^+ = \omega^-$).

Using the machinery of \S\ref{s: notation}, one can show that the pseudoblowups of the functions $v^{(Q)}$ along $\Gamma_d$ are homogeneous harmonic polynomials of degree $d$.

\begin{lemma}\label{l:vblowups} For every sequence $Q_j\in\Gamma_d$ converging to $Q\in\Gamma_d$ and every sequence $r_j\downarrow 0$, there exists a subsequence of $(Q_j,r_j)$, which we relabel, such that the functions $$v_{r_j}^{(Q_j)}(x):=\frac{v^{(Q_j)}(r_jx+Q) r_j^{n-2}}{\omega^-(B(Q_j,r_j))}$$ converge locally uniformly to a homogeneous harmonic polynomial $p$ of degree $d$. Moreover, $\nabla v^{(Q_j)}_{r_j}\stackrel{*}{\rightharpoonup} \nabla p$ with respect to the $L^\infty$ norm and strongly in $L^2_{\mathrm{loc}}$.\end{lemma}

\begin{proof} The lemma is valid under the assumption $\log h\in C(\partial\Omega)$. Let $Q_j$ be a sequence in $\Gamma_d$ converging to $Q$ in $\Gamma_d$ and let $r_j$ be a sequence of positive numbers converging to $0$. Because $h$ is continuous at $Q$ and $\excess(B(Q_j,r_j),\{Q\})\rightarrow 0$ as $j\rightarrow\infty$, \begin{equation*}\limsup_{j\rightarrow\infty} \left|\frac{\omega^-(B(Q_j,r_j))}{\omega^+(B(Q_j,r_j))}-h(Q)\right|
\leq \limsup_{j\rightarrow\infty} \dashint_{B(Q_j,r_j)} |h(x)-h(Q)|d\omega^+(x) =0\end{equation*} by the Radon-Nikodym theorem. Thus, \begin{equation}\label{eq:vblowup1} \lim_{j\rightarrow\infty} \frac{\omega^-(B(Q_j,r_j))}{\omega^+(B(Q_j,r_j))}=h(Q)=\lim_{j\rightarrow\infty} h(Q_j).\end{equation} Moreover, because $\log h$ is continuous, we have $h(Q)>0$ and \begin{equation}\label{eq:vblowup2} \sup_{j\geq 1} \frac{\omega^+(B(Q_j,r_j))}{\omega^-(B(Q_j,r_j))}<\infty.\end{equation} By Theorem \ref{t:blowups} and Theorem \ref{homogenouspseudoblowups}, there exists a subsequence of $(Q_j,r_j)$, which we relabel, such that the functions $$u_j^\pm(x) := \frac{u^\pm(r_jx+Q_j)r_j^{n-2}}{\omega^\pm(B(Q_j,r_j))}$$ converge locally uniformly to functions $u^\pm_\infty$ such that $p:=u^+_\infty-u^-_\infty$ is a homogeneous harmonic polynomial of degree $d$. By \eqref{eq:vblowup2}, it follows that $$ \frac{u^+(r_jx+Q_j)r_j^{n-2}}{\omega^-(B(Q_j,r_j))}$$ converges locally uniformly, as well. Therefore, by \eqref{eq:vblowup1}, \begin{align*}\lim_{j\rightarrow\infty} v_{r_j}^{(Q_j)}(x) &= \lim_{j\rightarrow\infty} \left( \frac{h(Q_j)u^+(r_jx+Q_j)r_j^{n-2}}{\omega^-(B(Q_j,r_j))} - \frac{u^-(r_jx+Q_j)r_j^{n-2}}{\omega^-(B(Q_j,r_j))}\right)\\
&= \lim_{j\rightarrow\infty}\left( \frac{u^+(r_jx+Q_j)r_j^{n-2}}{\omega^+(B(Q_j,r_j))} - \frac{u^-(r_jx+Q_j)r_j^{n-2}}{\omega^-(B(Q_j,r_j))}\right)= u^+_\infty(x)-u^-_\infty(x)=p(x), \end{align*} where the limits converge locally uniformly.

The statement regarding weak-star convergence is proven (in greater generality) in Remark \ref{r:weakstarconvergence}.
\end{proof}

\section{Almost monotonicity via Almgren's frequency functional}\label{formulasandcomputations}

\begin{definition}\label{almgrenfrequency} For all $f \in W^{1,2}_{\mathrm{loc}}(\R^n)$ and $Q\in\R^n$ such that $f(Q)=0$, define $$H(r, Q, f) := \int_{\partial B(Q,r)} f^2,\quad D(r,Q,f) := \int_{B(Q,r)} |\nabla f|^2,\quad N(r,Q,f) := \frac{rD(r,Q,f)}{H(r,Q,f)}.$$ The function $N(r,Q,f)$ is called \emph{Almgren's frequency functional}.
\end{definition}

\begin{remark}\label{rk:almgren} Almgren \cite{almgrenbigregularity} proved that when $f$ is harmonic, $N(\cdot, Q, f)$ is absolutely continuous and monotone increasing. In this case, $$N(0, Q, f):=\lim_{r\downarrow 0} N(r,Q,f)$$ is an integer and is the order to which $f$ vanishes at $Q$. If $p$ is a homogenous harmonic polynomial of degree $d$, then $N(r, 0, p) = d$ for all $r>0$.\end{remark}

We continue to adopt the notation set in \S\ref{sec:start}. In addition, given $Q\in \partial \Omega$, $r>0$, and $E\subset\partial\Omega$, we write
\begin{equation}\label{eqn-tt-4.0}
v^{(Q)}_{r}(x) :=  \frac{v^{(Q)}(rx + Q) r^{n-2}}{\omega^-(B(Q,r))}\quad\text{and}\quad \omega_{r,Q}^-(E) := \frac{\omega^-(rE + Q)}{\omega^-(B(Q,r))}.
\end{equation}
Choose a compact set $K\subset\partial\Omega$.  The estimates carried out below are uniform over $Q\in K$. To ease notation, we assume throughout this section that $0\in K$ and write $v:=v^{(0)}$ and
$\omega^-_{r, Q}=\omega^-_{r}$.

Ultimately, we would like to estimate $(d/dr)N(r, 0, v)$, but reach an immediate technical difficulty in that $v$ is merely Lipschitz, so $\nabla v$ is not defined everywhere. To address this, we work instead with the family of regularizations $v_\varepsilon = v*\varphi_\varepsilon$ ($0<\varepsilon\ll 1$), where $\varphi\ge 0$ is a $C^\infty$ approximation to the identity (i.e.~$\spt\varphi \subset B(0,1)$ and $\int_{\R^n} \varphi = 1$.) We abbreviate $H(r, 0, v_\varepsilon)$, $D(r,0,v_\varepsilon)$, and $N(r,0,v_\varepsilon)$ by $H_\varepsilon(r)$, $D_\varepsilon(r)$, and $N_\varepsilon(r)$, respectively.

\begin{remark}\label{convolutionunderblowup} For reference, let us recall how convolution behaves under a simple change of variables. For all $\varepsilon > 0$, $r > 0$, and $x, Q\in \mathbb R^n$, if $f_r(x) := f(rx + Q)$, then \begin{equation*}
(\varphi_\varepsilon *f)(rx+Q) = (\varphi_{\varepsilon/r}*f_r)(x)
\end{equation*}
\end{remark}

We now record several basic computations related to Almgren's frequency functional applied to the functions $v_\varepsilon$ and $v^{(Q)}$ (we note that \eqref{eq52}, \eqref{eq53}, and \eqref{eq54} appear, if not explicitly, then in spirit in \cite{almgrenbigregularity}).

\begin{lemma}\label{basicalmgrenfacts}
\begin{align} \label{eq52}
&D_\varepsilon(r) = \int_{\partial B(0,r)}v_\varepsilon(v_\varepsilon)_\nu\, d\sigma - \int_{B(0,r)} v_\varepsilon \Delta v_\varepsilon\\ \label{eq53}
&\frac{d}{dr}D_\varepsilon(r) = \frac{n-2}{r}\int_{B(0,r)}|\nabla v_\varepsilon|^2\,dx + 2\int_{\partial B(0,r)} (v_\varepsilon)_\nu^2 - \frac{2}{r}\int_{B(0,r)} \left\langle x, \nabla v_\varepsilon\right\rangle \Delta v_\varepsilon\, dx \\ \label{eq54}
&\frac{d}{dr}H_\varepsilon(r) = \frac{n-1}{r}H_\varepsilon(r) + 2\int_{\partial B(0,r)} v_\varepsilon (v_\varepsilon)_\nu\, d\sigma\\
\label{eq51}
&\lim_{r\downarrow 0} N(r,Q,v^{(Q)}) = d\quad\text{for all }Q\in \Gamma_d
\end{align}
\end{lemma}

\begin{proof}
Equation \eqref{eq52} follows from integration by parts. Equations \eqref{eq53} and \eqref{eq54} can be derived using the change of variables $y = x/r$. To establish \eqref{eq51}, assume that $Q\in\Gamma_d$ and pick any sequence $r_j \downarrow 0$. Then $$N(r_j, Q, v^{(Q)}) = \frac{\int_{B(0,1)} |\nabla v^{(Q)}_{r_j}|^2}{\int_{\partial B(0,1)} (v_{r_j}^{(Q)})^2}.$$ By Lemma \ref{l:vblowups}, we can pass to a subsequence of $r_j$, which we relabel, such that $v_{r_j}^{(Q)}$ converges locally uniformly to a homogeneous harmonic polynomial $p$, and moreover, since $\nabla v_{r_j}^{(Q)} \rightarrow \nabla p$ in $L^2_{\mathrm{loc}}$ then
$\int_{B(0,1)}|\nabla v_{r_j}^{(Q)}|^2 \rightarrow \int_{B(0,1)}|\nabla p|^2$  (see Remark \ref{r:weakstarconvergence}). It follows that $$\lim_{j\rightarrow \infty} N(r_j, Q,v^{(Q)}) = \lim_{j\rightarrow \infty} N(1, 0, v^{(Q)}_{r_j}) = N(1, 0, p)=d$$ by Remark \ref{rk:almgren}.
\end{proof}

The following lemma is a variation on \cite[Lemma 5.6]{engelsteintwophase}, which assumed $0\in\Gamma_1$. To extend this result to a general boundary point, we give an alternative ``rescaling argument''. In the remainder of this section, constants labeled $c$ or $C$ may depend on the NTA constants of $\Omega$, a choice of a compact set $K\subset\partial\Omega$, and $\|\log(h)\|_{C^{0,\alpha}(K)}$, but not on $u, Q, r, \varepsilon$.

\begin{lemma}\label{almostmonotonic}
Assume that $R>0$ and $\varepsilon\ll R$. There exists a function ${\bf E}(R, \varepsilon)$ such that
\begin{eqnarray}
 N_\varepsilon(R) + {\bf E}(R, \varepsilon)(R-r) &\geq& N_\varepsilon(r)\quad\text{for all }R/4 < r < R,\;\\
{\bf E}(R,\varepsilon)R &\leq& kR^{\alpha},
\end{eqnarray}
where $k > 0$ is a constant independent of $\varepsilon, R$.
\end{lemma}

\begin{proof}
Define ${\bf E}(R, \varepsilon) \colonequals \sup_{R/4 < r < R} (N_\varepsilon(r)')^-$, where prime denotes the derivative in $r$. The first claim is immediate.

For the second claim, recall the following formula for $N'_\varepsilon(r)$ from \cite[(5.4)]{engelsteintwophase}:
\begin{equation}\label{e: derivativeofN}\begin{aligned}
H_\varepsilon^2(r)N'_{\varepsilon}(r) &= 2r\left(\int_{\partial B_r}(v_\varepsilon)_\nu^2d\sigma \int_{\partial B_r} v_\varepsilon^2d\sigma -\left[\int_{\partial B_r} v_\varepsilon (v_\varepsilon)_\nu d\sigma\right]^2\right)\\ &\quad+ 2r\int_{B_r} v_\varepsilon \Delta v_\varepsilon dx \int_{\partial B_r} v_\varepsilon (v_\varepsilon)_\nu d\sigma -  2H_\varepsilon(r) \int_{B_r} \left\langle x, \nabla v_\varepsilon\right\rangle \Delta v_\varepsilon dx.
\end{aligned}
\end{equation} The difference in parenthesis on the right hand side of \eqref{e: derivativeofN} is positive by the Cauchy-Schwartz inequality. Therefore, \begin{equation}\label{e: estimatingnprime}(N'_{\varepsilon}(r))^- \leq 2\left|\frac{\int_{B_r} \left\langle x, \nabla v_\varepsilon\right\rangle \Delta v_\varepsilon dx}{H_{\varepsilon}(r)}\right| + 2\left|\frac{r\int_{B_r} v_\varepsilon \Delta v_\varepsilon dx \int_{\partial B_r} v_{\varepsilon} (v_{\varepsilon})_\nu d \sigma }{H_{\varepsilon}(r)^2}\right|.\end{equation}
Using the Caffarelli-Fabes-Mortola-Salsa (CFMS) type estimate on NTA domains (see \cite[Lemma 4.8]{jerisonandkenig}), it can be shown that if $\varepsilon < r/100$, then \begin{equation}\label{eq:lowerboundonH}H_\varepsilon(r) > c\frac{\omega^-(B(0,r))^2}{r^{n-3}}\end{equation} for some constant $c > 0$ independent of $r$ and $\varepsilon$. For details, see \cite[Lemma 5.5]{engelsteintwophase}.

We now estimate the two error terms on the right hand side of \eqref{e: estimatingnprime}. \smallskip

\noindent {\bf (A) Estimating $\frac{\int_{B_r} \left\langle x, \nabla v_\varepsilon\right\rangle \Delta v_\varepsilon dx}{H_{\varepsilon}(r)}$}: Since $\Delta v_\varepsilon = (\Delta v)*\varphi_\varepsilon$ in the sense of distributions, we can move the convolution from one term to the other:  $$\int_{B_r} \left\langle x, \nabla v_\varepsilon\right\rangle \Delta v_\varepsilon\, dx = \int  [(\chi_{B_r}(x)\left\langle x, \nabla v_\varepsilon \right\rangle)*\varphi_\varepsilon] \Delta  v\, dx.$$ Evaluate $\Delta v$, as in \eqref{distributionderivative}, to obtain \begin{equation}\label{e: estimatingEr} \begin{aligned}\left|\int_{B_r} \left\langle x, \nabla v_\varepsilon\right\rangle \Delta v_\varepsilon dx\right|=& \left|\int (\chi_{B_r}(x)\left\langle x, \nabla (v*\varphi_\varepsilon)(x) \right\rangle)*\varphi_\varepsilon \left(\frac{h(0)}{h(x)}-1\right)d\omega^-\right|\\
\leq& C r^{1+\alpha} \int (\chi_{B_r}(x)|\nabla (v*\varphi_\varepsilon)|)*\varphi_\varepsilon d\omega^-,\\
\end{aligned}\end{equation} where the last inequality follows from our assumption $\log(h) \in C^\alpha$ and the fact that $|x| < C(r + \varepsilon) < 2Cr$ on the domain of integration.

If $x = ry$, then $\nabla_x v(x) = \frac{1}{r}\nabla_y v(ry) = \frac{\omega^-(B(0,r))}{r^{n-1}}\nabla_y v_r(y)$. Changing variables, \eqref{e: estimatingEr} becomes \begin{equation}\begin{aligned}\label{e: estimatingEr2}\left|\int_{B_r} \left\langle x, \nabla v_\varepsilon\right\rangle \Delta v_\varepsilon dx\right| \leq& Cr^{1+\alpha} \frac{\left(\omega^-(B(0,r))\right)^2}{r^{n-1}} \int (\chi_{B_1}(y) |\nabla v_r*\varphi_{\varepsilon/r}|)*\varphi_{\varepsilon/r} d\omega^-_r\\
\leq& Cr^{\alpha} \frac{\left(\omega^-(B(0,r))\right)^2}{r^{n-2}} \omega^-_r(B(0, 1+\varepsilon/r))\\ \leq& \tilde{C}r^{\alpha} \frac{\left(\omega^-(B(0,r))\right)^2}{r^{n-2}}.\end{aligned}\end{equation} The penultimate inequality in \eqref{e: estimatingEr2} holds, because the $v_r$ are uniformly Lipschitz, and the last inequality holds, because $1+\varepsilon/r < 2$ and $\omega^-_r(B(0,2))$ is bounded uniformly in $r$.

Together, the upper bound \eqref{e: estimatingEr2} and the lower bound \eqref{eq:lowerboundonH} imply $$2\left|\frac{\int_{B_r} \left\langle x, \nabla v_\varepsilon\right\rangle \Delta v_\varepsilon dx}{H_{\varepsilon}(r)}\right| \leq Cr^{\alpha-1}.$$%
\smallskip

\noindent {\bf (B) Estimating $2r\int_{B_r} v_\varepsilon \Delta v_\varepsilon dx \int_{\partial B_r} v_{\varepsilon} (v_{\varepsilon})_\nu d \sigma$}: Arguing as in \eqref{e: estimatingEr} and \eqref{e: estimatingEr2}, we have \begin{equation}\label{e: estimateproductterm1} \begin{aligned}
\left|\int_{B_r} v_\varepsilon \Delta v_\varepsilon dx\right| &\leq Cr^\alpha \frac{\omega^-(B(0,r))^2}{r^{n-2}} \int \left|(\chi_{B_1}(v_r*\varphi_{\varepsilon/r}))*\varphi_{\varepsilon/r}\right| d\omega^-_r\\
&\leq Cr^\alpha \frac{\omega^-(B(0,r))^2}{r^{n-2}} \omega^-_r(B(0,1+\varepsilon/r)) \leq Cr^\alpha \frac{\omega^-(B(0,r))^2}{r^{n-2}}.
\end{aligned}
\end{equation}
Similarly, we can estimate \begin{equation}\label{e: estimateproductterm2}\begin{aligned}
\left|\int_{\partial B_r} v_{\varepsilon} (v_{\varepsilon})_\nu d \sigma\right| &= \frac{\omega^-(B(0,r))^2}{r^{2n-3}} r^{n-1} \int_{\partial B_1}\left| (v_r *\varphi_{\varepsilon/r})(\nabla v_r*\varphi_{\varepsilon/r})\cdot \nu\right| d\sigma\\
&\leq C\frac{\omega^-(B(0,r))^2}{r^{n-2}}.
\end{aligned}
\end{equation}
The lower bound \eqref{eq:lowerboundonH} combined with \eqref{e: estimateproductterm1} and \eqref{e: estimateproductterm2} yields $$\left|\frac{2r\int_{B_r} v_\varepsilon \Delta v_\varepsilon dx \int_{\partial B_r} v_{\varepsilon} (v_{\varepsilon})_\nu d \sigma }{H_{\varepsilon}(r)^2}\right| \leq Cr^{\alpha-1}.$$ This completes the proof of the lemma.
\end{proof}

We are ready to estimate the growth of $N(r)$. Compare Theorem \ref{t: growthrateforW} with Remark \ref{rk:almgren}.

\begin{theorem}[almost monotonicity]\label{t: growthrateforW} For all compact sets $K\subset \partial\Omega$, there exists $C>0$ such that for all $Q\in K$ and $0<r\leq 1$, \begin{equation}\label{e:growthforN} \frac{1}{r}\left(N(r, Q, v^{(Q)}) -N(0, Q, v^{(Q)})\right)  > -Cr^{\alpha-1},\end{equation}
where
\begin{equation}\label{eqn-tt-N}
N(0, Q, v^{(Q)})=\lim_{r\downarrow 0} N(r,Q,v^{(Q)}).
\end{equation}
\end{theorem}

\begin{proof}[Proof of Theorem \ref{t: growthrateforW}]
By Lemma \ref{basicalmgrenfacts}, $N(0, Q, v^{(Q)})$ as defined in \eqref{eqn-tt-N} exists. As such, for any $r > 0$, there exists $\tilde{r} \ll r$  such that $$|N(\tilde{r}, Q, v^{(Q)}) - N(0, Q, v^{(Q)})| < r^{\alpha}.$$ Now pick $\varepsilon \ll\tilde{r}$ small enough so that Lemma \ref{almostmonotonic} applies for $R \in (\tilde{r}, r)$ and such that $$|N(\tilde{r}, Q, v_\varepsilon^{(Q)}) - N(\tilde{r}, Q, v^{(Q)})| + |N(r, Q, v_\varepsilon^{(Q)}) - N(r, Q, v^{(Q)})|< r^{\alpha}.$$ This is possible, because $v^{(Q)}_\varepsilon  \rightarrow v^{(Q)}$ in $W^{1,2}$ as $\varepsilon \downarrow 0$.
Choose $j\in\mathbb{Z}$ such that $2^{-j}r \leq \tilde{r} < 2^{-j+1}r$. Then \begin{equation*}\begin{aligned} N(r, &Q, v_\varepsilon^{(Q)}) - N(\tilde{r}, Q, v_\varepsilon^{(Q)}) \\
&\geq N(2^{-j+1}r, Q, v_\varepsilon^{(Q)}) - N(\tilde{r}, Q, v_\varepsilon^{(Q)})  + \sum_{\ell=0}^{j-2} \left(N(2^{-\ell}r, Q, v_\varepsilon^{(Q)})-N(2^{-\ell-1}r, Q, v_\varepsilon^{(Q)})\right) \\
&\geq -{\bf E}(2^{-j+1}r, \varepsilon)(2^{-j+1}r-r')-\frac{1}{2}\sum_{\ell=0}^{j-2} {\bf E}(2^{-\ell}r, \varepsilon)2^{-\ell}r
\\ &\!\!\!\!\!\stackrel{\mathrm{Lem}\; \ref{almostmonotonic}}{\geq} -kr^{\alpha} \sum_{\ell=0}^{j-1}(2^{-\ell})^{\alpha/2} \geq -kr^\alpha \sum_{\ell=0}^\infty 2^{-\ell\alpha/2} =  -C_\alpha r^{\alpha}.\end{aligned}
\end{equation*}
Combining all the inequalities above, we conclude that \begin{equation*} N(r, Q, v^{(Q)}) - N(0, Q, v^{(Q)}) > -C r^{\alpha}.\qedhere\end{equation*}\end{proof}

\section{Bounding the growth of the Weiss type functional}\label{s:boundW}

 Recall the definition of the Weiss-type monotonicity formula $W_d(r, Q, f)$ from \S\ref{s: WeissMonotonicity}: $$W_d(r,Q,f)= \frac{1}{r^{n-2+2d}}\int_{B(Q,r)} |\nabla f|^2 dx -\frac{d}{r^{n-1+2d}}\int_{\partial B(Q,r)} f^2d\sigma.$$ In this section, we estimate the growth of $W_d(r,Q,v^{(Q)}_\varepsilon)$, and in particular, we prove that $|W_d(r,Q, v_\varepsilon^{(Q)})|$ is bounded above by a power of $r$. See Proposition \ref{p: Wdoesntgrowtoofast}.

The key idea in this section is that although $v^{(Q)}_\varepsilon$ is not harmonic, near $Q\in\partial\Omega$ the function ``almost minimizes'' the Dirichlet energy. This is related, but slightly different than almost minimizers in the sense of \cite{davidtoroalmostminimizers}, because the inequality that we establish only holds on sufficiently small balls centered at $Q\in \partial \Omega$.  However, as in \cite{DETAM}, this almost minimization allows us to apply monotonicity techniques. To establish almost minimization of $v^{(Q)}_\varepsilon$, it is crucial that we have estimates on the local dimension of the harmonic measure provided by Lemma \ref{growthofballsatsingularpoints}.

\begin{lemma}\label{l: almostminimizer}
For every $\alpha'\in(0,\alpha)$, integer $d\geq 1$, and every compact set $K$ in $\Gamma_d$, there exists a constant $C > 0$ with the following property. For all $Q \in K$, $r\in (0,1]$, $0<\varepsilon < r/100$, and $\zeta_{r,\varepsilon} \in W^{1,2}(B_1(Q))$ with $$\zeta_{r,\varepsilon}(\cdot+Q)|_{\partial B_1} = \frac{v_\varepsilon^{(Q)}(r\cdot+Q)}{r^d}\Big|_{\partial B_1},$$ we have $W_d(1, Q, \zeta_{r,\varepsilon})+ Cr^{\alpha'} \geq W_d(r, Q, v_\varepsilon^{(Q)}).$
\end{lemma}

\begin{proof}
Fix $r\in(0,1]$ and let $\zeta_\varepsilon(x+Q) = r^d\zeta_{r,\varepsilon}\left(\frac{x}{r}+Q\right)$, so that $W_d(1, Q, \zeta_{r,\varepsilon}) = W_d(r, Q, \zeta_\varepsilon)$ and $\zeta_\varepsilon|_{\partial B_r(Q)} = v_\varepsilon^{(Q)}|_{\partial B_r(Q)}$. If $\zeta^*_\varepsilon$ is the harmonic function on $B_r(Q)$ with boundary values equal to $\zeta_\varepsilon|_{\partial B_r(Q)}$, then $W_d(r, Q, \zeta_\varepsilon) \geq W_d(r, Q, \zeta^*_\varepsilon)$, because $\zeta^*$ minimizes the Dirichlet energy in its trace class. Thus, it suffices to prove \begin{equation*}
W_d(r, Q, \zeta^*_\varepsilon)+ Cr^{\alpha'} \geq W_d(r, Q, v_\varepsilon^{(Q)}).\end{equation*}
Equivalently, because $\zeta^*_\varepsilon|_{\partial B_r} = v_\varepsilon^{(Q)}|_{\partial B_r}$, it suffices to prove
\begin{equation}\label{eq:valmostminimizer}
\frac{1}{r^{n-2+2d}}\int_{B_r(Q)}|\nabla \zeta^*_\varepsilon|^2 + Cr^{\alpha'} \geq \frac{1}{r^{n-2+2d}}\int_{B_r(Q)} |\nabla v_\varepsilon^{(Q)}|^2.
\end{equation}  We compute \begin{equation}\label{eqn-tt-5.1A}
\begin{aligned}
\int_{B_r(Q)} |\nabla \zeta^*_\varepsilon|^2\ dx =& \int_{B_r(Q)} |\nabla v_\varepsilon^{(Q)} + \nabla (\zeta^*_\varepsilon - v_\varepsilon^{(Q)})|^2\ dx\\ \geq& \int_{B_r(Q)} |\nabla v^{(Q)}_\varepsilon|^2\ dx+ 2\int_{B_r(Q)} \nabla (\zeta^*_\varepsilon - v_\varepsilon^{(Q)})\cdot \nabla v_\varepsilon^{(Q)}\ dx\\
 =& \int_{B_r(Q)} |\nabla v^{(Q)}_\varepsilon|^2\ dx - 2\int_{B_r(Q)} (\zeta^*_\varepsilon - v_\varepsilon^{(Q)}) \Delta v_\varepsilon^{(Q)}\ dx.
\end{aligned}
\end{equation}
In light of \eqref{eqn-tt-5.1A}, to prove \eqref{eq:valmostminimizer}, we need only to show $$\left|\frac{1}{r^{n-2+2d}}\int_{B_r(Q)} (\zeta^*_\varepsilon - v_\varepsilon^{(Q)}) \Delta v_\varepsilon^{(Q)}\right| \leq Cr^{\alpha'}, \quad \forall 0 < \varepsilon < r/100.$$ By the maximum principle, $\zeta^*_\varepsilon$ on $B_r(Q)$ is less than the maximum of $v_\varepsilon^{(Q)}$ on $\partial B_r(Q)$, which in turn is less than $C\frac{\omega^-(B(Q,r))}{r^{n-2}}$ (see \eqref{eqn-tt-4.0}). This estimate, the fact that $\varepsilon < r$, the uniform doubling of $\omega^-$ and \eqref{e:boundondistderivative} yield \begin{equation}\label{e:lastamestimate}\begin{aligned} \left|\frac{1}{r^{n-2+2d}}\int_{B_r(Q)} (\zeta^*_\varepsilon - v_\varepsilon^{(Q)}) \Delta v_\varepsilon^{(Q)}\right| \leq& 2\frac{\omega^-(B(Q,r))}{r^{2n-4+2d}} \int_{B_r(Q)} |\varphi_\varepsilon*\Delta v^{(Q)}|\\ \leq&
C\left(\mathrm{osc}_{P\in B_{2r}(Q)} \frac{h(P)}{h(Q)}\right) \frac{\omega^-(B(Q,r))\omega^-(B(Q,2r))}{r^{2n-4+2d}}\\
\leq& Cr^\alpha \left(\frac{\omega^-(B(Q,r))}{r^{n-2+d}}\right)^2 \stackrel{\mathrm{Lem}\;\ref{growthofballsatsingularpoints}}{\leq} Cr^{\alpha'}.
\end{aligned}\end{equation} This completes the proof of the lemma.\end{proof}

\begin{proposition}\label{p: Wdoesntgrowtoofast}
For every $\alpha'\in(0,\alpha)$, integer $d\geq 1$, and every compact set $K$ in $\Gamma_d$, there exist constants $C > 0$ and $r_K > 0$ such that if $0<r \leq r_K$, then\begin{equation}\label{eq:Wdoesntgrowtoofast}
|W_d(r, Q, v^{(Q)})| \leq Cr^{\alpha'}.
\end{equation}
\end{proposition}

\begin{proof} Let $0<\alpha' <\alpha'' < \alpha$.
To show $W_d(r, Q, v^{(Q)}) \geq -Cr^{\alpha'}$, first observe that
\begin{equation}\label{e:relatewtohandn}W_d(r,Q,f) = \frac{H(r, Q, f)}{r^{n-1+2d}}(N(r,Q, f) - d).\end{equation}
From \eqref{e:relatewtohandn}, the estimate $$\frac{H(r,Q,v^{(Q)})}{r^{n-1+2d}} \leq C\left(\frac{\omega^-(B(Q,r))}{r^{n - 2+d}}\right)^2,$$
(which is obtained using a CFMS type estimate, see \cite[Lemma 4.8]{jerisonandkenig})
Theorem \ref{t: growthrateforW}, and Lemma \ref{growthofballsatsingularpoints} we get the lower bound on $W_d$. As an immediate consequence,
\begin{equation}\label{e:Watzero}
\liminf_{r\downarrow 0}W_d(r, Q, v^{(Q)}) \geq 0.
\end{equation}

To establish the upper bound on $W_d$, let $1 \geq r_0 > r_1 > 0$ and $\varepsilon < r_1/100$ so that we can apply Lemma \ref{l: almostminimizer} for all $r \in [r_1,r_0]$. Recall \eqref{formulaforwprime} with $f=v_\varepsilon^{(Q)}$. By ignoring the positive integral on the right hand side of \eqref{formulaforwprime}, we obtain
\begin{equation}\label{eqn-tt-5.5A}
\frac{d}{dr} W_d(r, Q, v^{(Q)}_\varepsilon) \geq -\frac{n+2d-2}{r}W_d(r, Q, v^{(Q)}_\varepsilon) +  \frac{n+2d-2}{r}W_d(1, 0, V_{r,Q,\varepsilon})=:I(r),
\end{equation}
where $V_{r,Q,\varepsilon}$ denotes the $d$-homogenous function about $Q$ that agrees with $r^{-d}v^{(Q)}_\varepsilon(rx+Q)$ on $\partial B_1$. Apply the epiperimetric inequality to $V_{r,Q, \varepsilon}$
(see Proposition \ref{epiperimetricinequality}) to conclude that $I$ as defined in \eqref{eqn-tt-5.5A} satisfies
\begin{equation}\label{eqn-tt-5.5B}
I(r)\geq -\frac{n+2d-2}{r}W_d(r, Q, v^{(Q)}_\varepsilon) +  \frac{n+2d-2}{r(1-\kappa)}W_d(1, 0, \zeta_{r,\varepsilon,Q}),
\end{equation}
where $\zeta_{r,\varepsilon,Q}\in W^{1,2}(B_1)$ denotes the harmonic extension of $V_{r,Q,\varepsilon}|_{\partial B_1}$, and 
 \begin{equation}\label{magickappa} \kappa = \frac{1}{n+2d-1}.\end{equation}
By Lemma \ref{l: almostminimizer} (applied with $\alpha''<\alpha$),
\begin{equation}\label{eqn-tt-5.5C}
I(r)\geq \frac{n+2d-2}{r}\left(\frac{\kappa}{1-\kappa}W_d(r, Q, v^{(Q)}_\varepsilon) - \frac{C}{1-\kappa}r^{\alpha''}\right),
\end{equation}
for $r\in [r_1, r_0]$ and $\varepsilon > 0$ as above. Observe that $(n+2d-2)\kappa/(1-\kappa)=1$ by \eqref{magickappa}. Thus \eqref{eqn-tt-5.5A} and \eqref{eqn-tt-5.5C} yield
\begin{equation} \label{differentialinequalityforepsilon}
\frac{d}{dr} W_d(r,Q,v^{(Q)}_\varepsilon) \geq I(r)\geq  \frac{1}{r}W_d(r,Q,v^{(Q)}_\varepsilon) - \frac{(n+2d-2)C}{(1-\kappa)r} r^{\alpha''}\end{equation} for $r\in[r_1,r_0]$ and $\varepsilon>0$ as above. Furthermore, $C>0$ is independent of $\varepsilon, r_0, r_1$, and the particular point $Q\in K$.

Fix $\tilde{C} > 0$ large to be chosen below and define $$\widetilde{W}(r) = W_d(r, Q, v^{(Q)}_\varepsilon) + \tilde{C}r^{\alpha''}.$$ Then \begin{equation}\label{e:differentialinequalityfortildeW}\begin{aligned} \frac{d}{dr}\widetilde{W}(r) =&  \frac{d}{dr}W(r, Q, v^{(Q)}_\varepsilon) + \alpha''\tilde{C}r^{\alpha''-1} \\
\stackrel{\eqref{differentialinequalityforepsilon}}{\geq}& \frac{1}{r}W_d(r, Q, v^{(Q)}_\varepsilon)+\left(\alpha''\tilde{C} - \frac{(n+2d-2)C}{1-\kappa}\right)r^{\alpha''-1}.
\end{aligned}\end{equation}

By the lower bound on the growth of $W_d$ from above \eqref{differentialinequalityforepsilon}, if $\varepsilon \ll r_1$ and $\tilde{C} = \tilde{C}(\kappa, \alpha, n, d) \gg 1$, then we have $\widetilde{W} > 0$ for $r \in [r_1, r_0]$ and $\varepsilon \ll r_1$ and $\alpha''\tilde{C} - \frac{(n+2d-2)C}{1-\kappa} \geq \alpha' \tilde{C}$.
With these assumptions, \eqref{e:differentialinequalityfortildeW} implies \begin{equation}\label{e:monoforwtilde}\frac{d}{dr}\widetilde{W}(r) \geq \frac{\alpha'}{r}\widetilde{W}>0.\end{equation} Thus, we have $$
\frac{d}{dr} \frac{\widetilde{W}(r)}{r^{\alpha'}} \geq 0\quad\text{for all }r \in [r_1, r_0]\text{ and } \varepsilon \ll r_1.$$ It follows that
\begin{equation}\label{e:lastdiffinequalityfortildeW}
\left(W_d(r_0,Q, v_\varepsilon^{(Q)}) + \tilde{C}r_0^{\alpha''}\right)r_0^{-\alpha'} \geq \left(W_d(r_1,Q, v_\varepsilon^{(Q)}) + \tilde{C}r_1^{\alpha''}\right)r_1^{-\alpha'}\end{equation} for all $\varepsilon \ll r_1$. Note that $\tilde{C}$ is independent of $r_0$, $r_1$, and $\varepsilon$ as long as $\varepsilon$ is small enough. Therefore,  letting $\varepsilon \downarrow 0$ and $r_0=1$ in \eqref{e:lastdiffinequalityfortildeW} (and noting that $W_d(1, Q, v_\varepsilon^{(Q)}) \leq \|v^{(Q)}\|^2_{\mathrm{Lip}(K)}$),
\begin{equation}\label{eqn-tt-5.10A}
(\tilde{C} + \|v^{(Q)}\|^2_{\mathrm{Lip}(K)})r_1^{\alpha'} =: C_{K} r_1^{\alpha'} \geq W_d(r_1, Q, v^{(Q)}). \qedhere
\end{equation}
\end{proof}

Recall Remark \ref{r:sharpregularity}, where we noted that the epiperimetric argument necessitates a loss of exponent. One sees this in the argument above, specifically in the discussion before \eqref{e:monoforwtilde} that one can only prove a rate of growth with exponent $\alpha'$ which is strictly less than the exponent $\alpha''$ that comes from the epiperimetric inequality. This issue is independent from the behavior of the quotient $r^{-(n+d-2)}\omega^{\pm}(B(Q,r))$ and occurs even if one wants to prove $C^{1,1}$ regularity of harmonic functions using the epiperimetric inequality.

\section{\texorpdfstring{$C^{1,\beta}$}{C1,beta} convergence to the blowup at singular points}\label{s: c1betaconvergence}

We now use Proposition \ref{p: Wdoesntgrowtoofast} to show that $v^{(Q)}$ converges at a H\"older rate to its blowups, uniformly over compact subsets of $\Gamma_d$. Let $K$ denote a compact subset of $\Gamma_d$, let $Q \in K$, and let $0 < \varepsilon \ll s < r \leq 1$. We abbreviate the blowups of $v_\varepsilon^{(Q)}$ and $\nabla v_\varepsilon^{(Q)}$ at $Q$ by $$Y(t,x)\equiv \frac{v_\varepsilon^{(Q)}(tx+Q)}{t^d},\quad Z(t,x)\equiv \frac{\nabla v_\varepsilon^{(Q)}(tx+Q)}{t^d}.$$
By the chain rule, $$\frac{d}{dt} Y(t,x)= x\cdot Z(t) - \frac{d}{t}Y(t,x).$$ Thus, by the fundamental theorem of calculus and H\"older's inequality, for $\gamma>0$ to be specified
\begin{equation}\begin{split} \left| Y(r,x)-Y(s,x)\right|^2&=\left(\int_s^r \left[ x\cdot Z(t,x) - \frac{d}{t}Y(t,x)\right]\,dt\right)^2\\
&\leq \left(\int_s^r t^{\gamma-1}\,dt \right) \left(\int_s^r \frac{1}{t^{\gamma-1}}\left[ x\cdot Z(t,x) - \frac{d}{t}Y(t,x)\right]^2\,dt\right)
\end{split}\end{equation}
Using Tonelli's theorem and a change of variables $y=tx+Q$, it follows that
\begin{equation}\begin{split}\label{e:boundblowupbyW}
\int_{\partial B(0,1)} &|Y(r,x)-Y(s,x)|^2\,d\sigma(x)\\
&\leq \frac{r^\gamma-s^\gamma}{\gamma}\int_{\partial B(0,1)} \int_s^r \frac{1}{t^{\gamma-1}}\left[ x\cdot Z(t,x) - \frac{d}{t}Y(t,x)\right]^2\,dt\,d\sigma(x)\\
&= \frac{r^\gamma-s^\gamma}{\gamma} \int_s^r \frac{1}{t^\gamma}\frac{1}{t^{n-2+2d}}\int_{\partial B(Q,t)} \left[\nu\cdot \nabla v_\varepsilon^{(Q)}(y)-\frac{d}{t}v_\varepsilon^{(Q)}(y)\right]^2\,d\sigma(y)dt,
\end{split}\end{equation}
where $\gamma \in (0,\beta)$ is arbitrary. On the other hand, let $0<\beta=\alpha'<\alpha''<\alpha<1$ and let $\tilde C$ be large enough so that both $\widetilde W(s)=W(s, Q, v^{(Q)}_\varepsilon) + \tilde{C}s^{\alpha''}>0$ and so that the final line of \eqref{pineapple} below holds (by Proposition \ref{p: Wdoesntgrowtoofast} such a $\tilde{C}$  need only to depend on $K, \beta, \alpha'', n, d$). Using \eqref{formulaforwprime} for all $t\in[s,r]$ and \eqref{differentialinequalityforepsilon}
we have
\begin{eqnarray}\label{pineapple}
&&\frac{1}{t^\gamma} \frac{1}{t^{n-2+2d}}\int_{\partial B_t(Q)} \left(\nu\cdot \nabla v_{\varepsilon}^{(Q)}(y) -\frac{dv_\varepsilon^{(Q)}(y)}{t}\right)^2\ d\sigma(y) \\
&&\qquad =
\frac{1}{t^\gamma}\frac{d}{dt}W_d(t,Q, v^{(Q)}_\varepsilon) -\frac{I(t)}{t^\gamma}\nonumber\\
&&\qquad\leq\frac{1}{t^\gamma}\left(\frac{d}{dt}W_d(t,Q, v^{(Q)}_\varepsilon)-\frac{1}{t}W_d(t,Q, v^{(Q)}_\varepsilon) +\frac{(n+2d-2)C}{(1-\kappa)t}t^{\alpha''}\right)\nonumber\\
&&\qquad\leq \frac{d}{dt}\left( \frac{W_d(t,Q, v^{(Q)}_\varepsilon)}{t^\gamma}\right) +\frac{d}{dt}\left(\frac{(n+2d-2)C}{(1-\kappa)(\alpha''-\gamma)}t^{\alpha''-\gamma}\right)\nonumber\\
& &\qquad\leq\frac{d}{dt}\left(\frac{W(t, Q, v^{(Q)}_\varepsilon) + \tilde{C}t^{\alpha''}}{t^\gamma}\right).\nonumber
\end{eqnarray}
We note that this estimate is valid, because $\varepsilon \ll s$. Chaining together \eqref{e:boundblowupbyW} and \eqref{pineapple}, we conclude that
\begin{equation}\label{e:boundblowupbyWfinal}\begin{aligned}\int_{\partial B(0,1)} \left| Y(r,x)-Y(s,x)\right|^2\,d\sigma(x) &\leq \frac{r^\gamma-s^\gamma}{\gamma}\left[\frac{W(t, Q, v^{(Q)}_\varepsilon) + \tilde{C}t^{\alpha''}}{t^\gamma}\right]_{t=s}^{t=r}\\
&< \frac{W(r, Q, v^{(Q)}_\varepsilon) + \tilde{C}r^{\alpha''}}{\gamma},\end{aligned}\end{equation} because $\widetilde W(s)=W(s, Q, v^{(Q)}_\varepsilon) + \tilde{C}s^{\alpha''}>0$ by the definition of $\tilde{C} > 0$. Finally, letting $\varepsilon \downarrow 0$ and invoking \eqref{eq:Wdoesntgrowtoofast} one more time, we conclude that
\begin{equation}\label{e:finalboundonblowup} \int_{\partial B(0,1)} \left| Y(r,x)-Y(s,x)\right|^2\,d\sigma(x) \leq C_{\beta, K} r^\beta\quad\text{for all }Q\in K, 0<s<r\leq 1,\end{equation} where $\beta$ is any exponent less than $\alpha$.

With \eqref{e:finalboundonblowup} in hand, we can now obtain uniqueness of pseudoblowups along $\Gamma_d$ and existence of the density.

\begin{theorem}\label{l:blowupsunique} For all $1\leq d\leq d_0$ and $Q\in\Gamma_d$, there exists a unique $d$-homogenous harmonic polynomial $p=p^{(Q)}$ such that for all sequences $Q_i\in\Gamma_d$ with $Q_i\rightarrow Q$ and $r_i\downarrow 0$ the blowups $v_{r_i}^{(Q_i)} \rightarrow p^{(Q)}$ as $i\rightarrow\infty$ uniformly on compact sets. Moreover,
\begin{equation}\label{density-exists}
 D^{n-2+d}(\omega^-,Q):=\lim_{r\downarrow 0} \lim_{  P\in\Gamma_d, P\to Q}\frac{\omega^-(B(P,r))}{r^{n-2+d}}\quad\text{exists}
 \end{equation} and $D^{n-2+d}(\omega^-,Q)\in(0,\infty).$\end{theorem}

\begin{proof}Let $Q\in\Gamma_d$. We first prove the existence of $D^{n-2+d}(\omega^-,Q)$. Consider arbitrary sequences $r_j\downarrow 0$ and $Q_j\in\Gamma_d$ with $Q_j\to Q$ such that
$$\delta:=\lim_{j\rightarrow\infty} \frac{\omega^-(B(Q_j,r_j))}{r_j^{n-2+d}}\in[0,\infty]
$$
exists. By passing to a subsequence, we know that $v^{(Q_j)}_{r_j} \rightarrow p$ for some $d$-homogenous harmonic polynomial $p$ by Lemma \ref{l:vblowups}.
By definition, $$ \frac{v^{(Q_j)}(r_jx+Q_j)}{r_j^d}  = \frac{\omega^-(B(Q_j,r_j))}{r_j^{n-2+d}} v^{(Q_j)}_{r_j}(x).$$
Thus, applying \eqref{e:finalboundonblowup} to $r$ and $r_j$ we get that (pick $j$ large enough so that $r_j < r$)
\begin{equation}\label{e:ktblowupsbounded} \int_{\partial B_1} \left| \frac{v^{(Q_j)}(rx+Q_j)}{r^d} - \frac{\omega^-(B(Q_j,r_j))}{r_j^{n-2+d}} v^{(Q_j)}_{r_j}(x)\right|^2 \leq C r^\beta.\end{equation}
By Remark \ref{r:ball-norm}, $\sup_{B(0,1)} |v^{(Q_j)}_{r_j}|\simeq 1$ for all sufficiently large $j$. Since $v^{(Q_j)}(rx+Q_j)=h(Q_j)u^+(rx+Q_j)- u^-(rx+Q_j)$ and $u^\pm$ and $h$ are continuous then
$\lim_{j\to\infty}v^{(Q_j)}(rx+Q_j)=h(Q)u^+(rx+Q)- u^-(rx+Q)=v^{(Q)}(rx+Q)$ uniformly on $x\in\partial B_1$. This yields $\delta<\infty$, otherwise we reach a contradiction by letting $j\rightarrow \infty$ in \eqref{e:ktblowupsbounded}.
If $\delta = 0$, then letting $j \rightarrow \infty$ in \eqref{e:ktblowupsbounded}, we have
\begin{equation}\label{e:dontdegenerate}
 \left(\frac{\omega^-(B(Q,r))}{r^{n-2+d}}\right)^2 \int_{\partial B_1} \left|v_r^{(Q)}(x)\right|^2=\int_{\partial B_1} \left| \frac{v^{(Q)}(rx+Q)}{r^d}\right|^2 < C r^\beta.\end{equation} Hence
 $$
   \int_{\partial B_1} \left|v_r^{(Q)}(x)\right|^2 < C\left(\frac{r^{n-2+d+\beta/2}}{\omega^-(B(Q,r))}\right)^2 \stackrel{r\downarrow 0}{\rightarrow} 0$$ by Lemma \ref{growthofballsatsingularpoints}. This contradicts the fact that $\sup_{x \in B(0,1)}|v_r^{(Q)}(x)| \simeq 1$. Thus $\delta\in(0,\infty)$. Since the sequences $r_j\downarrow 0$ and $Q_j\in\Gamma_d$ with $Q_j\to Q$ were arbitrary this shows that
\begin{equation}\label{density-tt-1}
0<\lambda(Q):=\liminf_{r\downarrow 0, P\in\Gamma_d, P\to Q}\frac{\omega^-(B(P,r))}{r^{n-2+d}}\leq \Lambda(Q):=\limsup_{r\downarrow 0, P\in\Gamma_d, P\to Q}\frac{\omega^-(B(P,r))}{r^{n-2+d}}<\infty.
\end{equation}
Let $r_j\downarrow 0$, $Q_j\in\Gamma_d$ with $Q_j\to Q$ and $s_\ell\downarrow 0$,  $P_\ell\in\Gamma_d$ with $P_\ell\to Q$ such that
\begin{equation}\label{density-tt-2}
\lambda(Q)=\lim_{\ell\rightarrow\infty} \frac{\omega^-(B(P_\ell,s_\ell))}{s_\ell^{n-2+d}}\qquad\hbox{ and }\qquad \Lambda(Q)=\lim_{j\rightarrow\infty} \frac{\omega^-(B(Q_j,r_j))}{r_j^{n-2+d}}
\end{equation}
Modulo passing to a subsequence (which we relabel), we know that $v^{(Q_j)}_{r_j} \rightarrow p$ for some $d$-homogenous harmonic polynomial $p$ and $v^{(P_\ell)}_{s_\ell} \rightarrow q$ for some $d$-homogenous harmonic polynomial $q$ by Lemma \ref{l:vblowups}. Our goal is to show that $\lambda(Q)=\Lambda(Q)$ and $p=q$. In order to do this, we estimate the following quantity for $s_j,\, r_j\le r$.  Using
\eqref{e:ktblowupsbounded}, we have
 \begin{eqnarray}\label{e:ktblowupsthesame-tt}
&& \int_{\partial B_1} \left| \frac{\omega^-(B(Q_j,r_{j}))}{r_{j}^{n-2+d}} v^{(Q_j)}_{r_{j}}(x)- \frac{\omega^-(B(P_j,s_{j}))}{s_{j}^{n-2+d}} v^{(P_j)}_{s_{j}}(x)\right|^2 \\
&&\qquad\lec  \int_{\partial B_1} \left| \frac{\omega^-(B(Q_j,r_{j}))}{r_{j}^{n-2+d}} v^{(Q_j)}_{r_{j}}(x)-\frac{v^{(Q_j)}(rx+Q_j)}{r^d}\right|^2\nonumber\\
&&\qquad +  \int_{\partial B_1} \left|\frac{v^{(Q_j)}(rx+Q_j)}{r^d}-\frac{v^{(P_j)}(rx+P_j)}{r^d}\right|^2\nonumber\\
&&\qquad +  \int_{\partial B_1} \left| \frac{\omega^-(B(P_j,s_{j}))}{s_{j}^{n-2+d}} v^{(P_j)}_{r_{j}}(x)-\frac{v^{(P_j)}(rx+P_j)}{r^d}\right|^2\nonumber\\
&&\le Cr^\beta + C  \int_{\partial B_1} \left|\frac{v^{(Q_j)}(rx+Q_j)}{r^d}-\frac{v^{(P_j)}(rx+P_j)}{r^d}\right|^2\nonumber.
 \end{eqnarray}
 Letting $j\to\infty$ in \eqref{e:ktblowupsthesame-tt} yields
 \begin{equation}\label{uniqueness-tt-1}
  \int_{\partial B_1} \left|\Lambda(Q)p-\lambda(Q)q\right|^2\le Cr^\beta\quad\text{for all $r>0$.}
 \end{equation}
Thus, letting $r\to 0$ in \eqref{uniqueness-tt-1}, we have
  \begin{equation}\label{uniqueness-tt-2}
  \int_{\partial B_1} \left|\Lambda(Q)p-\lambda(Q)q\right|^2=0.
 \end{equation}
 Since $p$ and $q$ are homogenous polynomials of degree $d$, \eqref{uniqueness-tt-2} implies that $\Lambda(Q)p=\lambda(Q)q$. By Remark \ref{r:ball-norm}, we conclude that $\Lambda(Q)=\lambda(Q)$ and $p(x) = q(x)$.
\end{proof}

\begin{definition}\label{def:ptilde}
For each $1\leq d\leq d_0$ and $Q \in \Gamma_d$, let $\tilde p^{(Q)}$ denote the homogeneous harmonic polynomial of degree $d$ defined by $$\tilde{p}^{(Q)} \equiv \left(\lim_{r\downarrow 0} \frac{\omega^-(B(Q,r))}{r^{n-2+d}}\right) p^{(Q)},$$ where $p^{(Q)}$ is the homogeneous harmonic polynomial which is the unique limit of $v_r^{(Q)}$.
\end{definition}

Note that $$\frac{v^{(Q)}(rx+Q)}{r^d} \rightarrow \tilde{p}^{(Q)}(x).$$

\begin{remark}\label{rem:uniformity}
Note that \eqref{density-exists} ensures that
 $$0 < \lim_{r\downarrow 0}\frac{\omega^-(B(Q,r))}{r^{n-2+d}} < \infty\quad\text{ for all }Q\in \Gamma_d.$$
Moreover, the proof of Theorem \ref{l:blowupsunique} gives a locally uniform bound on the density:
\begin{equation}\label{uniform-density-bound}
0 < \lim_{r\downarrow 0} \inf_{Q\in K\cap \Gamma_d}  \frac{\omega^-(B(Q,r))}{r^{n-2+d}} \leq  \lim_{r\downarrow 0} \sup_{Q\in K\cap \Gamma_d}  \frac{\omega^-(B(Q,r))}{r^{n-2+d}} < \infty\quad\text{for every }K \subset\subset \Gamma_d.
\end{equation}
\end{remark}

Now that we know the blowup of $v^{(Q)}$ is unique, a power rate of convergence follows immediately from \eqref{e:finalboundonblowup}.
\begin{corollary}\label{p:blowupatholderrate}
For every $\beta\in(0,\alpha)$, integer $d\geq 1$, and every compact set $K$ in $\Gamma_d$, there exists a constant $C > 0$ such that for all  $Q\in K$ and $0<r < 1$,
 \begin{equation}\label{holderlinfinity}\left\|\frac{v^{(Q)}(r\cdot+Q)}{r^d} - \tilde{p}^{(Q)}\right\|^2_{L^2(\partial B_1(0))} < Cr^\beta.\end{equation}
\end{corollary}

Finally, we prove that the blowups change in a H\"older continuous manner.

\begin{corollary}\label{c:blowupschangeholder}
For $\beta\in(0,\alpha)$, integer $d\geq 1$, and every compact set $K$ in $\Gamma_d$, there exists a constant $C > 0$ such that for all $Q_1,Q_2\in K$,
\begin{equation}\label{blowupschangeholdercenteredatzero}
\| \tilde{p}^{(Q_1)} - \tilde{p}^{(Q_2)}\|_{C(B_1)} \leq C|Q_1-Q_2|^{\alpha\beta/(\beta+2)}.
\end{equation} Moreover, for all $0<r\leq 1$ and $x\in B_{r/2}(Q_1)\cap B_{r/2}(Q_2)$,
\begin{equation}\label{blowupschangeholdermovingcenter}
|\tilde{p}^{(Q_1)}(x-Q_1) - \tilde{p}^{(Q_2)}(x-Q_2)| \leq Cr^d(r^{\beta/2} + |Q_1-Q_2|^{\alpha}).
\end{equation}
\end{corollary}

We will end up applying this inequality when

\begin{proof} When $Q_1, Q_2$ are far apart, \eqref{blowupschangeholdercenteredatzero} holds with $C$ large, since $\| \tilde{p}^Q\|_{L^\infty}$ is uniformly bounded over $Q\in K\cap \Gamma_d$  by Remark \ref{rem:uniformity}. Thus, to prove \eqref{blowupschangeholdercenteredatzero}, we may assume that $Q_1, Q_2$ are close enough together so that $|Q_1-Q_2|^{1/2d}< 1/2$. Pick $\rho > 0$ such that $|Q_1-Q_2|^{1-\alpha} \geq \rho \geq |Q_1- Q_2|$. Because $h\in C^{0,\alpha}(\partial\Omega)$,
\begin{equation}\label{6.11A}
\begin{aligned}
&\left|\frac{v^{(Q_1)}(\rho x+Q_1)}{\rho^d} - \frac{v^{(Q_2)}(\rho x+Q_2)}{\rho^d}\right| \\ &\qquad\leq \frac{C|Q_1-Q_2|^\alpha}{\rho^d}\sup_{y\in B_{C\rho}(Q_1)}\left(|Q_1 - Q_2|^{1-\alpha}|\nabla u^{\pm}| + |u^+|\right)\\
&\qquad\leq C\frac{|Q_1-Q_2|^\alpha \omega^{\pm}(B(Q_1, C\rho))}{\rho^{d+n-2}}\left(\frac{|Q_1-Q_2|^{1-\alpha}}{\rho} + 1\right) \\
&\qquad\stackrel{\eqref{uniform-density-bound}}{\leq} C|Q_1-Q_2|^\alpha\left(1+\frac{|Q_1 - Q_2|^{1-\alpha}}{\rho}\right) \leq C\left(\frac{|Q_1-Q_2|}{\rho}\right),
\end{aligned}
\end{equation} where the second inequality follows because $u^{\pm}_j$ is bounded in $L^\infty$ (recall Remark \ref{r:weakstarconvergence}) and the last inequality follows from our choice of $\rho$. Combining \eqref{holderlinfinity} and \eqref{6.11A}, we obtain
\begin{equation}\label{p-estimate}
\| \tilde{p}^{(Q_1)} - \tilde{p}^{(Q_2)}\|^2_{L^2(\partial B_1(0))}\le C\rho^\beta + C\left(\frac{|Q_1-Q_2|^\alpha}{\rho}\right)^2.
\end{equation}
Since ${\tilde p}^{(Q_i)}$ for $i=1,2$ are homogeneous harmonic polynomials of bounded degree, \eqref{p-estimate} plus equivalence of norms on finite dimensional vector spaces yield
\begin{equation}\begin{split}\label{p-estimate-cont}
\| \tilde{p}^{(Q_1)} - \tilde{p}^{(Q_2)}\|^2_{L^\infty( B_1(0))} &\le C \| \tilde{p}^{(Q_1)} - \tilde{p}^{(Q_2)}\|^2_{L^2(\partial B_1(0))}\\
&\le C\rho^\beta + C\left(\frac{|Q_1-Q_2|^\alpha}{\rho}\right)^2.\nonumber
\end{split}\end{equation}
Taking $\rho^{\beta/2} = |Q_1-Q_2|^{\alpha}/r$, we obtain \eqref{blowupschangeholdercenteredatzero} (note that if $\rho^{\beta/2 + 1} = |Q_1 - Q_2|^\alpha$, then $|Q_1- Q_2|^{1-\alpha} > \rho > |Q_1- Q_2|$ so our choices are compatible).

We turn to proving \eqref{blowupschangeholdermovingcenter} and let $0 < r \leq 1$. Estimate \eqref{holderlinfinity} implies $$\fint_{B_r(Q)} |v^{(Q)}(z) - \tilde{p}^{(Q)}(z-Q)|^2 \leq Cr^{2d+\beta}.$$ (To pass from integrating on the sphere to integrating on the ball, integrate \eqref{holderlinfinity} in $r$.) Furthermore, for $z\in B_{r}(Q_1)\cap B_{r}(Q_2)$,
$$|v^{(Q_1)}(z)-v^{(Q_2)}(z)| \leq C|Q_1-Q_2|^\alpha |u^+(z)| \leq C|Q_1-Q_2|^\alpha\dfrac{\omega^+(B(Q_1,r))}{r^{n-2}}.$$ Combining these estimates and noting that $x\in B_{r/2}(Q_1)\cap B_{r/2}(Q_2)$ implies that $B_{r/2}(x) \subset B_r(Q_1)\cap B_r(Q_2)$, we obtain $$\fint_{B_{r/2}(x)}|\tilde{p}^{(Q_1)}(y-Q_1) - \tilde{p}^{(Q_2)}(y-Q_2)|^2 \leq C|Q_1-Q_2|^{2\alpha} \left(\dfrac{\omega^+(B(Q_1,r))}{r^{n-2}}\right)^2 + Cr^{2d+\beta}.$$ Because $Q_1\in K$, we have $r^{-(n-2)}\omega^+(B(Q_1,r)) \leq Cr^d$ by Remark \ref{rem:uniformity}. Taking square roots yields \eqref{blowupschangeholdermovingcenter} (since the average of the $L^2$ norm on the set $B_{r/2}(x)$ bounds the value at $x$ for any harmonic function).\end{proof}

\section{Higher order rectifiability of the singular set}\label{s: rectifiability}

In this section, we complete our proof of Theorem \ref{t:main} and show that the singular set is contained in a countable union of $C^{1,\beta}$ manifolds of dimension $\leq n-3$. To state our result rigorously, we need to introduce a notion of a ``dimension" of a point.

\begin{definition}[cf. {\cite[Definition 1.3.7]{garofaloandpetrosyan}}] \label{ddependence}
For each $Q\in \partial \Omega \backslash \Gamma_1$ define the {\it dimension} of the point, $d(Q)$, by $$d(Q) = \dim\{\zeta\in \mathbb R^n \mid \zeta \cdot \nabla p^{(Q)}(x) = 0\text{ for all }x\in \mathbb R^n\}.$$ Sometimes we abuse terminology and refer to $d(Q)$ as the {\it dimension} of $p^{(Q)}$.

For $d,j \in \mathbb N$, define $\Gamma_d^j = \{Q\in \Gamma_d\mid d(Q) = j\}$.
\end{definition}

\begin{remark}\label{propertiesofd}
If $Q \in\partial\Omega\setminus \Gamma_1$, then $d(Q) \leq n-3$, because $p^{(Q)}$ is a homogeneous harmonic polynomial of degree at least 2 and the set $\RR^n\setminus\{p^{(Q)}=0\}$ has two connected components. See the introduction of \cite{BETHarmonicpoly} for details.
\end{remark}

Theorem \ref{t:main} follows from the following proposition.

\begin{proposition}\label{singularsetinmanifolds}
For every $\beta\in(0,\alpha)$, $d\in \mathbb N$, and $0 \leq j \leq n-3$, the set $\Gamma_d^j$ is contained in a countable union of $j$-dimensional $C^{1,\beta}$ manifolds.  \end{proposition}

From here the proof of Proposition \ref{singularsetinmanifolds} follows a now classical argument using the implicit function theorem (cf. the proof of Theorem 1.3.8 in \cite{garofaloandpetrosyan} for a similar approach
to study the singular set in the thin obstacle problem). We include the main statements here for completeness, but omit some details when the two-phase problem for harmonic measure provides no additional complications.

An extension of the classical Whitney extension theorem with respect to arbitrary modulus of continuity was established by Glaeser \cite{glaeser}. We use the following special case,  which provides criteria to guarantee that an extended function has H\"older continuous derivatives.

\begin{theorem}[see e.g.~{\cite[Theorem 15]{lewisconepoints}}] \label{holderwhitneyextension}
Let $\beta \in (0,1), k, \ell, n \in \mathbb N, A\subset \mathbb R^n$ be closed, and for each $a\in A$ a polynomial $P_a: \mathbb R^n \rightarrow \mathbb R^\ell$ such that $\deg P_a \leq k$. Define for $K \subseteq A, r > 0$ and multi-index $\alpha$ with $0 \leq |\alpha| \leq k$, $$ \rho_\alpha(K,r) = \sup\left\{\frac{|D^\alpha P_b(b) - D^\alpha P_a(b)|}{|a-b|^{k-|\alpha|}}: a,b \in A,\; |a-b| \leq r\right\}.$$ If for each compact $K\subset A$ and each multi-index $\alpha$ with $0 \leq |\alpha| \leq k$ $$\rho_\alpha(K,r) \leq Cr^\beta$$ then there exists $F\in C^{k,\beta}_{\mathrm{loc}}(\mathbb R^n; \mathbb R^\ell)$ such that for all $a\in A$ and multi-index $\alpha$, $D^\alpha F(a) = D^\alpha P_a(a)$.
\end{theorem}

For $m,\, d\in\mathbb N$ define the set $$K_{d,m} = \{Q\in B_m(0)\cap \partial \Omega \mid \frac{1}{m} \rho^d \leq \sup_{\partial B_\rho(Q)} |u^\pm(x)| \leq m \rho^d, \forall \rho \in [0,1]\}.$$ Clearly $K_{d,m}$ is compact. Furthermore, by Remark \ref{rem:uniformity} since $\sup_{B_r(Q)} u^{\pm} \leq C\frac{\omega^-(B_r(Q))}{r^{n-2}}$ we know that $K_{d,m} \subset \Gamma_d$ and $\Gamma_d = \bigcup_m K_{d,m}$.

\begin{lemma}\label{blowupssatisfywhitney}
Fix a $d\in \mathbb N$. For any $m \in \mathbb N$ the polynomials $\{\tilde{p}^{(Q)}(x-Q) \mid Q \in K_{m,d}\}$ satisfy the conditions of Theorem \ref{holderwhitneyextension}, with $k = d$, for any $\beta \in (0,\alpha/2)$. In particular, there is a function $F \in C^{d,\beta}(B_m(0))$ such that for all $Q\in K_{m,d}$ and multi-indices $\chi$ with $|\chi| \leq d$ we have $D^\chi F(Q) = D^\chi p^{(Q)}(0)$.
\end{lemma}

\begin{proof}
Fix $Q_1, Q_2 \in K$ and consider the harmonic function $H(x) \colonequals \tilde{p}^{(Q_1)}(x-Q_1) -\tilde{p}^{(Q_2)}(x-Q_2)$. In this new notation, to satisfy conditions of Theorem \ref{holderwhitneyextension}, we must show $|D^\chi H(Q_1)| \leq C|Q_1-Q_2|^{\beta +d-|\chi|}$. Note, it suffices to prove this for $Q_1, Q_2$ close together, say $|Q_1-Q_2| <1/1000$. For $Q_1, Q_2$ further apart, the upper estimate is rendered trivial in light of the fact that that $\|\tilde{p}^{(Q_1)}\|_{C^{\ell, \beta}(B_1)} \leq C_{\ell,\beta}\|\tilde{p}^{(Q_1)}\|_{L^\infty(B_1)}$ for all $d$-homogenous harmonic polynomials (and that $\|\tilde{p}^{(Q_1)}\|_{L^\infty(B_1)}$ is bounded uniformly over $Q_1 \in K$ by Remark \ref{rem:uniformity}).

We have the classical interior harmonic estimate $$\sup_{B(Q_1, 2|Q_1-Q_2|)} |D^\chi H| \leq C|Q_1-Q_2|^{-|\chi|} \sup_{B(Q_1, 4|Q_1-Q_2|)} |H|.$$ As $B(Q_1, 4|Q_1-Q_2|) \subset B(Q_2, 10|Q_1-Q_2|)$, \eqref{blowupschangeholdermovingcenter} tells us $$\sup_{B(Q_1, 4|Q_1-Q_2|)} |H| \leq C|Q_1-Q_2|^d(|Q_1-Q_2|^\beta + |Q_1-Q_2|^{\alpha}) \leq C|Q_1-Q_2|^{\beta + d}.$$ Putting these estimates together, \begin{equation*}|D^\chi H(Q_1)| \leq \sup_{B(Q_1, 2|Q_1-Q_2|)} |D^\chi H| \leq C|Q_1-Q_2|^{\beta + d-|\chi|},\qedhere\end{equation*}
which guarantees that the hypothesis of Theorem \ref{holderwhitneyextension} hold.
\end{proof}

\begin{proof}[Proof of Proposition \ref{singularsetinmanifolds}]
Fix, $m, d, j \in \mathbb N$ and let $x_0 \in \Gamma_d^j \cap K_{m,d}$. We show that there is an open neighborhood $\mathcal O$ of $x_0$ such that $\mathcal O \cap ( \Gamma_d^j \cap K_{m,d})$ is contained in a $C^{1,\beta}$ manifold of dimension $j$. By compactness we can cover all of $\Gamma_d^j \cap K_{m,d}$ with finitely many of these such neighborhoods. Finally, $\Gamma_d^j \subset \Gamma_d = \bigcup_m K_{m,d}$ so taking a countable union of finite covers gives us the desired collection of $j$-dimensional $C^{1,\beta}$ manifolds.

By Lemma \ref{blowupssatisfywhitney} and Theorem \ref{holderwhitneyextension}, associated to $K_{m,d}$ is a $C^{1,\beta}$ function $F$ such that for each multi-index $\chi$ with $|\chi| \leq d$ we have $D^\chi F(Q) = D^\chi \tilde{p}^{(Q)}(0)$, where $Q \in K_{m,d}$. As $x_0 \in \Gamma_d^j$ we claim that there are $n-j$ multi-indices, $\{\chi_i\}_{i=1}^{n-j}$, with $|\chi_i| = d-1$ such that $v_i \colonequals \nabla D^{\chi_i} \tilde{p}^{x_0}(0)$ is a set of $n-j$ linearly independent vectors. For details see the proof of Theorem 1.3.8 in \cite{garofaloandpetrosyan}.

Define $\tilde{F}: \mathbb R^n \rightarrow \mathbb R^{n-j}$ by $\tilde{F}(x) =  (D^{\chi_1} F(x), D^{\chi_2}F(x),\ldots, D^{\chi_{n-j}}F(x))$. One can check that $K_{m,d} \subset \{\tilde{F} = 0\}$. On the other hand, by the claim above, $D F(x_0)$ has rank $n-j$ and therefore, by the implicit function theorem, there is a neighborhood of $x_0$ on which $\{\tilde{F}= 0\}$ is actually a $j$-dimensional $C^{1,\beta}$ manifold.
\end{proof}

\appendix

\section{An epiperimetric inequality for harmonic functions}

Recall that for any $f\in W_{loc}^{1,2}(\RR^n)$, $Q \in \RR^n$, $r > 0$, and $d\in(0,\infty)$, we define
\begin{equation*} W_d(r, Q, f) := \frac{1}{r^{n-2+2d}}\int_{B(Q,r)} |\nabla f|^2 dx -\frac{d}{r^{n-1+2d}}\int_{\partial B(Q,r)} f^2d\sigma.\end{equation*}

\begin{proposition}[an epiperimetric inequality for harmonic functions]\label{epiperimetricinequality-stadedagain}
For every integer $n\geq 2$ and real number $d>0$, there exists $\kappa\in(0,1)$ such that if $u \in W^{1,2}(B(Q,r))$ is homogeneous of degree $d$ about $Q$ and $f$ denotes the harmonic extension of $u|_{\partial B(Q,r)}$ to $B(Q,r)$, then
\begin{equation}\label{eq:epiperimetricinequality-statedagain} W_d(r,Q,f) \leq (1-\kappa)W_d(r,Q,u).\end{equation}
In fact, when $d$ is an integer we can take $\kappa = 1/(n + 2d-1)$.
\end{proposition}

\begin{proof} Without loss of generality, we may assume that $Q=0$ and $r=1$. Let $u\in W^{1,2}(B(0,1))$ be homogeneous of degree $d\in(0,\infty)$. Then $c:= u|_{\partial B(0,1)}\in L^2(\partial B_1(0))$, and thus, the harmonic extension $f$ of $c$ to $B(0,1)$ is well-defined.

Expand $c=\sum_{j=1}^\infty c_j\phi^j$, where $\{\phi^j: j\geq 1\}$ denotes a sequence of spherical harmonics that form an orthonormal basis of $L^2(\partial B(0,1))$. For each $j\geq 1$, let $d_j=\deg \phi^j$. Then $$u(r,\theta) = \sum_{j=1}^\infty c_j r^d\phi^j(\theta)\quad\text{and}\quad f(r,\theta) = \sum_{j=1}^\infty c_j r^{d_j} \phi^j(\theta).$$
On one hand, since $f$ is harmonic, $$\int_{B_1(0)} |\nabla f|^2\ dx = \int_{\partial B_1(0)} (x\cdot \nabla f)f\ d\sigma = \sum_{j=1}^\infty d_j c_j^2,$$ which implies \begin{equation}\label{wofharmonicextension} W_d(1, 0, f) = \sum_{j=1}^\infty (d_j-d)c_j^2.
\end{equation}
On the other hand, $$|\nabla u|^2 = (\partial_r u)^2 +\frac{1}{r^2}(\partial_\theta u)^2 = r^{2d-2}\left[\left(\sum_{j=1}^\infty dc_j\phi^j(\theta)\right)^2 + \left(\sum_{j=1}^\infty c_j \partial_\theta\phi^j(\theta)\right)^2\right].$$
Hence $$\int_{B_1(0)}|\nabla u|^2\ dx = \frac{d^2}{n+2d-2}\sum c_j^2 + \frac{1}{n+2d-2}\sum c_j^2 \int_{\partial B_1(0)} (\partial_\theta \phi^j)^2\ d\theta.$$
Note that $\phi^j$ is an eigenfunction of the Laplace-Beltrami operator on the sphere with eigenvalue $\lambda_j := d_j(n+d_j-2)$. Thus, $$\int_{\partial B_1(0)} (\partial_\theta\phi^j)^2\ d\theta = \lambda_j \int_{\partial B_1(0)} (\phi^j)^2\ d\theta = d_j(n+d_j-2).$$
All together, \begin{equation*} \int_{B_1(0)}|\nabla u|^2\ dx = \sum \frac{d^2+d_j(n+d_j-2)}{n+2d-2}\,c_j^2.\end{equation*} Therefore, \begin{equation}\label{wofhomogeneousextension}
W_d(1, 0,u) = \sum_{j=1}^\infty \frac{d^2 +d_j(n+d_j-2)-d(n+2d-2)}{n+2d-2}\,c_j^2.\end{equation}
Compare \eqref{wofharmonicextension} and \eqref{wofhomogeneousextension}. To complete the proof, it suffices to find $\kappa \in (0,1)$ such that $$j-d \leq (1-\kappa)\frac{d^2 +j(n+j-2)-d(n+2d-2)}{n+2d-2}\quad\text{ for all $j\in \mathbb N$,}$$ or equivalently, \begin{equation}\label{jd} j-d \leq (1-\kappa)\frac{(j-d)(n+j+d-2)}{n+2d-2}\quad\text{for all }j\in\mathbb{N}.\end{equation} When $j\leq d$, inequality \eqref{jd} holds for any $\kappa\geq 0$. When $j>d$, or equivalently, when $j\geq \lfloor d\rfloor+1$, inequality \eqref{jd} holds provided that $$1\leq (1-\kappa) \frac{n+j+d-2}{n+2d-2}.$$ Thus, it suffices to choose $\kappa\in(0,1)$ so that $$1= (1-\kappa)\frac{n+\lfloor d\rfloor + 1+d-2}{n+2d-2}.$$ Therefore, \eqref{eq:epiperimetricinequality-statedagain} holds with \begin{equation}\kappa:= \frac{1+\lfloor d\rfloor - d}{n+\lfloor d\rfloor+d-1}. \end{equation} Note that when $d \in \mathbb Z$ we get the desired formula for $\kappa$.
\end{proof}

\bibliography{sreg-refs}{}
\bibliographystyle{amsbeta}

\end{document}